\documentclass[11pt]{amsart}
\usepackage{extsizes}
\usepackage{amsfonts}
\usepackage{latexsym}
\usepackage{color}
\numberwithin{equation}{section}

\input{xypic}
\xyoption{all}

\usepackage{graphicx,amsmath,amssymb,amsthm,array,enumerate,mhsetup,mathtools,comment,multicol,tabularx}

\newtheorem{theorem}{Theorem}[section] 
 
 \newtheorem{innercustomgeneric}{\customgenericname}
 \providecommand{\customgenericname}{}

\newtheorem{lemma}[theorem]{Lemma}

\theoremstyle{definition}
\newtheorem{conjecture}{Conjecture}[section]
\newtheorem{remark}{Remark}

\newcommand{\norm}[1]{\ensuremath{N(\ideal{#1})}}

\newcommand{\authorfootnotes}{\renewcommand\thefootnote{\@fnsymbol\c@footnote}}%
\makeatother

\begin{document}

\def\Ext{{\rm Ext}}
\def\Hom{{\rm Hom}}
\def\Ind{{\rm Ind}}

\def\res{\mathop{Res}}

\def\GL{{\rm GL}}
\def\SL{{\rm SL}}
\def\SO{{\rm SO}}
\def\O{{\rm O}}

\def\R{\mathbb{R}}
\def\C{\mathbb{C}}
\def\Z{\mathbb{Z}}
\def\Q{\mathbb{Q}}
\def\A{\mathbb{A}}

\def\w{\wedge}

\def\Cat{\mathcal{C}}
\def\HC{{\rm HC}}
\def\HCat{\Cat^\HC}
\def\proj{{\rm proj}}

\def\to{\rightarrow}
\def\To{\longrightarrow}

\def\1{1\!\!1}
\def\dim{{\rm dim}}

\def\th{^{\rm th}}
\def\isom{\approx}

\def\CE{\mathcal{C}\mathcal{E}}
\def\E{\mathcal{E}}

\def\dis{\displaystyle}
\def\f{{\bf f}}                 
\def\g{{\bf g}}
\def\T{{\rm T}}              
\def\omegatil{\tilde{\omega}}  
\def\H{\mathcal{H}}            
\def\Dif{\mathfrak{D}}      
\def\W{W^{\circ}}           
\def\Whit{\mathcal{W}}      
\def\ringO{\mathcal{O}}     
\def\S{\mathcal{S}}      
\def\M{\mathcal{M}}      
\def\K{{\rm K}}          
\def\h{\mathfrak{h}} 
\def\norm{{\rm N}}       
\def\trace{{\rm Tr}} 
\def\ctilde{\tilde{C}}

\title{Additive Sums of Generalized Shifted Ternary Divisor Function}

\author{Alia Hamieh}
\address[Alia Hamieh]{University of Northern British Columbia, Department of Mathematics and Statistics, 3333 University Way, Prince George, BC V2N4Z9, Canada}
\email{alia.hamieh@unbc.ca}

\author{Do Nhat Tan Vo}
\address[Do Nhat Tan Vo]{University of Lethbridge, Department of Mathematics and Computer Science, 4401 University Dr W, Lethbridge, AB T1K3M4, Canada}
\email{donhattan.vo@uleth.ca}

\thanks{Research of Alia Hamieh was partially supported by the NSERC Discovery Grant RGPIN-2018-06313.}
\thanks{The research of Do Nhat Tan Vo was partially supported by the University of Lethbridge (ULGRA) and by his supervisors H. Kadiri and N. Ng through their respective NSERC Discovery Grants RGPIN-2020-06731, RGPIN-2020-06032, and RGPIN-2026-05876.}


\date{\today}

\begin{abstract}
In this paper, we prove a conditional result that yields an asymptotic formula with an explicit main term and a power-saving error term for the generalized shifted ternary additive divisor sum given by  
\[D_{\mathcal{I},\mathcal{J}}(X,1)=\sum_{n \leq X} \tau_{\mathcal{I}}(n) \tau_{\mathcal{J}}(n + 1),\]
where $\mathcal{I}=\{\alpha_1,\alpha_2,\alpha_3\}$ and $\mathcal{J} = \{\beta_1, \beta_2, \beta_3\}$, with $|\alpha_i|, |\beta_i|\leq \delta$ for $i=1,2,3$, and $0<\delta<1/44$. The proof of our main theorem relies on assuming an averaged level of distribution for the generalized ternary additive divisor function $\tau_{\mathcal{J}}$ in arithmetic progressions up to the level $2/3$.
\end{abstract}

\maketitle

\section{Introduction}
The Riemann zeta function
\begin{equation*}
\zeta(s)=\sum_{n=1}^{\infty}\frac{1}{n^s}
=\prod_p(1-p^{-s})^{-1},
\qquad \Re(s)>1,
\end{equation*}
admits an analytic continuation to $\mathbb{C}\setminus\{1\}$, with a simple pole of residue $1$ at $s=1$, and satisfies the functional equation
\begin{equation*}
    \pi^{-\frac{s}{2}}\Gamma\left(\frac{s}{2}\right)\zeta(s)=\pi^{-\frac{1-s}{2}}\Gamma\left(\frac{1-s}{2}\right)\zeta(1-s).
\end{equation*}
By its Euler product, $\zeta(s)\neq0$ for $\Re(s)>1$; the negative even integers are the trivial zeros of $\zeta(s)$, produced by the functional equation. The Riemann hypothesis asserts that every nontrivial zero of $\zeta(s)$ lies on the critical line $\Re(s)=1/2$, and a great deal of work in analytic number theory is conditional on it and its generalizations. A closely related conjecture is the Lindel\"of hypothesis, which states that
\begin{equation}\label{eq:lindelof}
    \zeta(1/2+it)=\mathcal{O}\left((1+|t|)^\epsilon\right),
\end{equation}
for every $\epsilon>0$. Backlund \cite{Backlund} showed that the Riemann hypothesis implies \eqref{eq:lindelof}, and the best known unconditional bound toward it is due to Bourgain \cite{bourgain}, namely $\zeta(1/2+it)=\mathcal{O}\left((1+|t|)^{13/84+\epsilon}\right)$. In pursuit of the Lindel\"of hypothesis, one is led to study $|\zeta(1/2+it)|$ on average over intervals of length $T$, through the moments
\begin{equation*}
    I_k(T)=\int_0^T\left|\zeta\left(\tfrac12+it\right)\right|^{2k}\,dt.
\end{equation*}
Indeed, $\zeta(1/2+it)=\mathcal{O}(t^\epsilon)$ for every $\epsilon>0$ if and only if $I_k(T)=\mathcal{O}(T^{1+\epsilon})$ for every $k\in\mathbb{N}$ and every $\epsilon>0$. Beyond this equivalence, moments of $\zeta(s)$ (and of $L$-functions more generally) have become an extensive area of study in their own right, with consequences for the size, zero distribution, value distribution, and non-vanishing of these functions. In 1918, Hardy and Littlewood \cite{Hardy} proved that $I_1(T)\sim T\log T$. Ingham \cite{second} sharpened this in 1926 to
\begin{equation}\label{eq:ingham-k1}
I_{1}(T)=TP_{1}(\log T)+\mathcal{O}\left(T^{1/2+\epsilon}\right),
\end{equation}
with $P_1$ a polynomial of degree $1$, and also treated the fourth moment, showing
\begin{equation}\label{eq:ingham-k2}
    I_2(T)=\frac{T}{2\pi^2}\left(\log T\right)^4+\mathcal{O}\left(T(\log T)^3\right).
\end{equation}
In 1979, Heath-Brown \cite{HB} improved the error term in $I_2(T)$ to
\begin{equation}\label{eq:HB-k2}
I_{2}(T)=TP_{4}(\log T)+\mathcal{O}\left(T^{7/8+\epsilon}\right),
\end{equation}
with $P_4$ of degree $4$. Aside from refinements of the error terms in \eqref{eq:ingham-k1} and \eqref{eq:HB-k2}, these remain the only unconditional asymptotic formulas known for $I_k(T)$, leaving $k\geq3$ entirely conjectural.  Keating and Snaith \cite{keating-snaith}, using random matrix theory, conjectured that
\begin{equation}\label{eq:keating-snaith}
    I_k(T)\sim c_kT(\log T)^{k^2},
\qquad T\to\infty,
\end{equation}
for explicit constants $c_k$. The only values of $k$ for which \eqref{eq:keating-snaith} is known unconditionally remain $k=1,2$, as above. For $k\geq3$, substantial partial progress exists: assuming the Riemann hypothesis, Soundararajan \cite{Sounda} obtained $I_k(T)\ll_kT(\log T)^{k^2+\epsilon}$, which Radziwiłł \cite{Radz} sharpened to the sharp exponent for $k<2.18$, and Harper \cite{harper2013sharp} then established $I_k(T)\ll_kT(\log T)^{k^2}$ for all $k>0$ under the Riemann hypothesis; unconditionally, matching lower bounds of the correct order were established by Radziwiłł and Soundararajan \cite{RS}, and subsequently by Heap and Soundararajan \cite{HS} in the remaining small-moment range. Underlying both the unconditional formulas above and the conjectural picture for larger $k$ is a common technical object: additive correlations of divisor functions. Indeed, in Ingham's and Heath-Brown's proofs of \eqref{eq:ingham-k2} and \eqref{eq:HB-k2}, one is led to 
\begin{equation}\label{eq:correlation}
D_k(X,h)=\sum_{n\leq X}\tau_k(n)\tau_k(n+h),
\qquad
\tau_k(n)=\sum_{n_1\cdots n_k=n}1,
\end{equation}
$h\in\mathbb{N}$, $h\neq0$. Heath-Brown's improvement of $I_2(T)$, in particular, rests on an asymptotic formula for $D_2(X,h)$ building on Estermann's work \cite{Estermann}. Beyond $k=2$, however, the sums $D_k(X,h)$ are less understood: even an asymptotic formula for $D_3(X,h)=\sum_{n\leq X}\tau_3(n)\tau_3(n+h)$ remains open.

In studying $I_k(T)$, it is sometimes more convenient to work with the shifted moment
\begin{equation*}
    I_{\mathcal{I},\mathcal{J}}(\omega)=\int_{-\infty}^\infty\left(\prod_{j=1}^k\zeta\left(\frac{1}{2}+\alpha_j+it\right)\zeta\left(\frac{1}{2}+\beta_j-it\right)\right)\omega(t)\,dt,
\end{equation*}
where $\mathcal{I}=\{\alpha_1,\ldots,\alpha_k\}$, $\mathcal{J}=\{\beta_1,\ldots,\beta_k\}$ are multisets of complex shifts and $\omega$ is a smooth, compactly supported weight on an interval of length a constant multiple of $T$. This approach, introduced by Ingham \cite{second} for the second moment, is technically easier than working with $I_k(T)$ directly. Ng \cite{Nathan} applied it to the sixth moment, proving an asymptotic formula for $I_{\mathcal{I},\mathcal{J}}(\omega)$ conditional on an additive divisor conjecture for a smoothed version of $D_{\mathcal{I},\mathcal{J}}(X,h)$ below. At zero shifts this gives, for some $\nu\in[1/2,2/3)$ and $C\geq0$,
\begin{equation*}
    I_3(T)=TP_9(\log T)+\mathcal{O}\left(T^{\frac{3\nu/2+C+1}{2+C}+\epsilon}\right),
\qquad T\to\infty,
\end{equation*}
i.e. the Keating--Snaith conjecture \eqref{eq:keating-snaith} for $k=3$, conditionally. Since the main conditional input in Ng's result is an additive divisor conjecture for this shifted correlation, we are motivated to study it directly:
\begin{equation*}
D_{\mathcal{I},\mathcal{J}}(X,h)
=
\sum_{n\leq X}
\tau_{\mathcal{I}}(n)\tau_{\mathcal{J}}(n+h),
\qquad
\tau_{\mathcal{I}}(n)
=
\sum_{n_1\cdots n_k=n}
n_1^{-\alpha_1}\cdots n_k^{-\alpha_k}.
\end{equation*}
Before turning to the shifted setting, we first review what is known unconditionally when all shifts vanish. Write $D_{k,l}(X,h)=\sum_{n\leq X}\tau_k(n)\tau_l(n+h)$, so that $D_k(X,h)=D_{k,k}(X,h)$ in the notation of \eqref{eq:correlation}. Essentially every unconditional asymptotic formula for $D_{k,l}(X,h)$ is confined to the case $l=2$: the pair correlation $D_{2,2}(X,h)=D_2(X,h)$ was treated by Ingham, Estermann, Heath-Brown, Deshouillers-Iwaniec, Motohashi, and Duke-Friedlander-Iwaniec \cite{Ingham,Estermann,HB,Deshouillers,Motohashi1,duke}, while various cases of $D_{k,2}(X,h)$ for $k\geq3$ were treated by Hooley, Heath-Brown, Linnik, Motohashi, Drappeau, and Topaco\u{g}ullar{\i} \cite{HL,Brown,Linnik,Motohashi,Drappeau,Topacogullari}.
Once $k=l\geq3$, however, no unconditional asymptotic formula is known at all, and one must turn instead to conjectures. Ivi\'{c} \cite{Ivic} conjectured in 1997 that $D_k(X,h)\sim XP_{2k-2}(\log X;h)$ for a suitable polynomial $P_{2k-2}$ of degree $2k-2$. Conrey and Gonek \cite{Gonek} proposed a more explicit smooth main term $m_k(X,h)$ together with the error bound $\mathcal{O}\left(X^{1/2+\epsilon}\right)$, uniformly for $1\leq h\leq X^{1/2}$. In 2019, Ng and Thom \cite{Thom} gave a probabilistic-heuristic derivation of a simplified additive divisor conjecture: for $\epsilon>0$, $k\geq2$, and $1\leq h\leq X^{1-\epsilon}$,
\begin{equation*}
    D_k(X,h)\sim\frac{c_k(h)}{((k-1)!)^2}X(\log X)^{2k-2},
\qquad X\to\infty,
\end{equation*}
where $c_k(h)=C_kf_k(h)$ is given explicitly in terms of an Euler product $C_k$ and a multiplicative function $f_k$ built from the divisor function $\tau_{k-1}$. Tao \cite{Tao} independently derived, via a different probabilistic heuristic, an alternative local-factor expression $c_k(h)=\prod_p\mathfrak{S}_{k,k,p}(h)$. Ng and Thom verified in \cite[Section 4]{Thom} that the two expressions for $c_k(h)$ agree. These heuristics, together with the special case $k=l=3$, motivate the search for an unconditional or conditional proof of an asymptotic formula for $D_3(X,h)$. Nguyen \cite{Nguyen} takes up exactly this problem, conditional on the following conjecture regarding the averaged level of distribution of $\tau_k(n)$ in arithmetic progressions.

\begin{conjecture}\label{D}
Let $\epsilon>0$. Then, for any $k \geq 1$, we have, uniformly in $1\leq h \leq X^{\frac{k-1}{k}}$, the upper bound
\begin{equation*}
    \sum_{q\leq X^{\frac{k-1}{k}}}\left|\sum_{\substack{n\leq X\\n\equiv h \bmod q}}\tau_k(n)-\frac{1}{\varphi\left(\frac{q}{(h,q)}\right)}\sum_{\substack{n\leq X \\ \left(n,\frac{q}{(h,q)}\right)=1}}\tau_k(n)\right|\ll_\epsilon X^{\frac{1}{2}+\epsilon},
\end{equation*}
as $X\rightarrow \infty$, where the implied constant is independent of $h$ and depends only on $\epsilon$.
\end{conjecture}
Nguyen proves \cite[Theorem~1]{Nguyen} the following result.
\begin{theorem}\label{thm:Nguyen}
Assume Conjecture~\ref{D} for $k=3$. Let $\epsilon>0$ be given. Then, for any composite shift $1 \leq |h| \leq X^{\frac{2}{3}}$, we have
\begin{equation*}
    D_{3}(X,h)=M_{3}(X,h)+E_{3}(X,h),\quad \text{as}\;\;X \rightarrow \infty,
\end{equation*}
where
\begin{align*}
    M_{3}(X,h)&=3\;\operatorname{Res}_{\substack{s=1\\w_1=w_2=0}}\left(\frac{(X+h)^sX^{\frac{1}{3}(w_1+2w_2)}}{sw_1w_2}\zeta^3(s)\zeta(w_1+w_2+1)\zeta(w_2+1)A_1(s,w_1,w_2)\right)\\
    &\quad-3\;\operatorname{Res}_{\substack{s=1\\w_2=1,w_1=0}}\left(\frac{X^{\frac{1}{3}(w_1+2w_2+s)}}{sw_1w_2}\zeta^3(s)\zeta(w_1+w_2+1-s)\zeta(w_2+1-s)A_2(s,w_1,w_2)\right)\\
    &\quad+\;\operatorname{Res}_{\substack{s=1\\w_1=w_2=1}}\left(\frac{X^{\frac{1}{3}(w_1+w_2+s)}}{sw_1w_2}\zeta^3(s)\zeta(w_1+1-s)\zeta(w_2+1-s)A_3(s,w_1,w_2)\right)\\
    &\quad+\mathcal{O}\left(X^{0.897}\right),
\end{align*}
where $A_1$, $A_2$, and $A_3$ are the arithmetic factors defined in \cite[Theorem~1]{Nguyen}. Moreover,
$E_3(X,h)\ll_\epsilon X^{\frac12+\epsilon}$.
\end{theorem}
\begin{remark}
We have studied the proof of \cite[Theorem~1]{Nguyen} but were not able to verify the stated error term of size $\mathcal{O}\left(X^{0.897}\right)$. By adapting the same line of argument used in \cite[Section~3]{Nguyen}, we were able to obtain an error term of size $\mathcal{O}\left(X^{0.924+\epsilon}\right)$. The details are given in the proof of the main theorem in the second author's Master's thesis. We were also unable to verify that the stated asymptotic formula in \cite[Theorem~1]{Nguyen} holds throughout the claimed range of shifts. Since the argument for general $h$ appears to require additional care, we restrict our attention in this paper to the case $h=1$.
\end{remark}
Under Conjecture \ref{D} for all $k$, Nguyen also establishes the leading-order main term in $M_{k,l}(X,h)$ for general $k,l$ and composite shift $1\leq h\leq X^{(k-1)/k}$, and verifies that this recovers the Ng-Thom \cite{Thom} and Tao \cite{Tao} heuristics above in the special case $k=l=3$; see \cite[Theorem~2, Corollary~2]{Nguyen}. To place Conjecture \ref{D} in context, Table
\ref{tab:divisor_distribution} collects the best known exponents of distribution for $\tau_k$ in arithmetic progressions. Even for $k=3$, the strongest known unconditional exponent remains well below the
level $2/3$ required by Conjecture \ref{D}. A useful survey of results on the distribution of generalized divisor functions in arithmetic progressions may be found in the introduction of \cite{Nguyen2}.

\begin{table}[h]
\centering
\caption{Known results on the exponent of distribution of $\tau_k$
in arithmetic progressions}
\label{tab:divisor_distribution}
\begin{tabular}{|c|c|c|}
\hline
$k$ & Exponent of distribution & References \\
\hline
$2$ & $\frac{2}{3}$ &
Heath-Brown \cite{HB} (1979) \\
\hline
$3$ & $\frac{1}{2}+\frac{1}{30}$ &
Sharma \cite{Sharma} (2024) \\
\hline
$4$ & $\frac{1}{2}$ &
Linnik \cite{Linnik} (1961) \\
\hline
$k\geq4$ & $\frac{8}{3k+4}$ &
Lavrik \cite{Lavrik} (1965) \\
\hline
$5$ & $\frac{9}{20}$ &
Friedlander-Iwaniec \cite{FI} (1985) \\
\hline
$6$ & $\frac{5}{12}$ &
Friedlander-Iwaniec \cite{FI} (1985) \\
\hline
$k\geq7$ & $\frac{8}{3k}$ &
Friedlander-Iwaniec \cite{FI} (1985) \\
\hline
$k\geq5$ & $\geq\frac12$ (open) &
Open problem \\
\hline
\end{tabular}
\end{table}

For $k=3$, Sharma's exponent $\frac12+\frac1{30}$ is proved for square-free moduli and for certain prime-power moduli; more precisely, for $q=p^\gamma$ with $p$ odd and $\gamma\geq 28$ \cite{Sharma}. It improves the earlier exponents obtained by Friedlander-Iwaniec \cite{Friedlander}, Heath-Brown \cite{Brown}, and Fouvry-Kowalski-Michel \cite{Fouvry}.

We extend Nguyen's result by presenting a conditional proof of the complete asymptotic expansion of $D_{\mathcal{I},\mathcal{J}}(X,1)$ for $\mathcal{I}=\{\alpha_1,\alpha_2,\alpha_3\}$, $\mathcal{J}=\{\beta_1,\beta_2,\beta_3\}$ with $0\neq|\alpha_i|,|\beta_i|\leq\delta$, $i=1,2,3$, for some positive constant $\delta$, assuming the following shifted analogue of Conjecture \ref{D} at $k=3$.

\begin{conjecture}\label{conjecture}
Let $\epsilon>0$, $0<\delta<1$, and
$\mathcal J=\{\beta_1,\beta_2,\beta_3\}$ with
$|\beta_i|\leq\delta$ for $i=1,2,3$. Define
\[E(Y;1,q)
=\sum_{\substack{n\leq Y\\ n\equiv1\bmod q}}
\tau_{\mathcal J}(n)-\frac{1}{\varphi(q)}
\sum_{\substack{n\leq Y\\ (n,q)=1}}
\tau_{\mathcal J}(n).
\]
Then there exists $\Delta<1$ such that
\[\sum_{q\leq X^{2/3}}
\max_{1\leq Y\leq X}
|E(Y;1,q)|\ll_{\epsilon,\delta}
X^{\Delta+\epsilon},\]
as $X\to\infty$.
\end{conjecture}

\begin{remark}
Nguyen's conjecture is stated for a fixed upper endpoint, whereas we use a maximal form. The maximal form provides the uniformity needed in the argument below, since the shifted weights require the Abel summation formula and lead to progression sums with varying upper endpoint $Y=qt+1$.
\end{remark}

Consider the multisets $\mathcal I=\{\alpha_1,\alpha_2,\alpha_3\}$ and
$\mathcal J=\{\beta_1,\beta_2,\beta_3\}$. For $j\in\{1,2,3\}$, define
\begin{align}\label{eq:A1-intro}
A_j^{(1)}(w_1,w_2)
&=
\prod_p
\left(1-\frac{1}{p^{1+w_1+w_2+\alpha_1-\alpha_3}}\right)
\left(1-\frac{1}{p^{1+w_2+\alpha_2-\alpha_3}}\right)
\nonumber\\
&\quad\times
\Bigg(
1+
\prod_{\substack{i=1\\i\neq j}}^3
\left(1-\frac{1}{p^{1-\beta_j+\beta_i}}\right)
\Bigg(
\frac{1}{p^{1+w_1+w_2+\alpha_1-\alpha_3}-1}
+
\frac{1}{p^{1+w_2+\alpha_2-\alpha_3}-1}
\nonumber\\
&\hspace{14em}
+
\frac{1}{
\left(p^{1+w_1+w_2+\alpha_1-\alpha_3}-1\right)
\left(p^{1+w_2+\alpha_2-\alpha_3}-1\right)}
\Bigg)
\Bigg).
\end{align}
We further set
\begin{align}\label{eq:A2-intro}
A_j^{(2)}(w_1,w_2)
&=
A_j^{(1)}
\left(w_1,w_2-1+\alpha_3+\beta_j\right),\end{align}
and \begin{align}\label{eq:A3-intro}
A_j^{(3)}(w_1,w_2)
&=
A_j^{(1)}
\left(w_1-w_2,w_2-1+\alpha_3+\beta_j\right).
\end{align}
For any function $F(\alpha_1,\alpha_2,\alpha_3)$, we write
\[
\sum_{\mathrm{cyc}} F(\alpha_1,\alpha_2,\alpha_3)
=
F(\alpha_1,\alpha_2,\alpha_3)
+
F(\alpha_2,\alpha_3,\alpha_1)
+
F(\alpha_3,\alpha_1,\alpha_2).
\]
In this notation, each cyclic permutation is applied to all occurrences
of $\alpha_1,\alpha_2,\alpha_3$ in the corresponding summand, including
those occurring in the arithmetic factors $A_j^{(1)}$.

We define
\begin{align}\label{eq:Main-final}
Main_{\mathcal I,\mathcal J}(X)
&=\sum_{\mathrm{cyc}}\sum_{j=1}^3
\frac{X^{1-\beta_j-\alpha_3}}
{1-\beta_j-\alpha_3}
\prod_{\substack{i=1\\i\neq j}}^3
\zeta(1-\beta_j+\beta_i)\nonumber\\
&\qquad\times
\zeta(1+\alpha_1-\alpha_3)
\zeta(1+\alpha_2-\alpha_3)
A_j^{(1)}(0,0).
\end{align}

\begin{theorem}\label{maintheorem}
Assume Conjecture~\ref{conjecture} with
$\Delta<\frac{61}{66}$. Let $\epsilon>0$ and let
$\mathcal I=\{\alpha_1,\alpha_2,\alpha_3\}$ and
$\mathcal J=\{\beta_1,\beta_2,\beta_3\}$, where
\[
\alpha_i\neq\alpha_j,\qquad
\beta_i\neq\beta_j\qquad (i\neq j),
\]
and
\[
0\neq |\alpha_i|,|\beta_i|\leq\delta,
\qquad
0<\delta<\frac1{44}.
\]
Then, as $X\to\infty$,
\[D_{\mathcal I,\mathcal J}(X,1)
=Main_{\mathcal I,\mathcal J}(X)
+\mathcal O_{\epsilon,\delta}
\left(
X^{\frac{61}{66}+\frac{10}{3}\delta+\epsilon}
\right),
\]
where $Main_{\mathcal I,\mathcal J}(X)$ is defined in
\eqref{eq:Main-final}.
\end{theorem}

The rest of the article is organized as follows. In
Section~\ref{sec:lemmas}, we collect the auxiliary lemmas needed for the proof of Theorem~\ref{maintheorem}, including a combinatorial decomposition of $\tau_{\mathcal I}(n)$ generalizing a lemma of Hooley, a Dirichlet series identity for $\tau_{\mathcal J}(n)$ restricted to $(n,h)=1$, a subconvexity bound for $\zeta(s)$ on vertical lines, and a multidimensional truncated Perron formula. Section~\ref{sec:main-proof} contains the proof of
Theorem~\ref{maintheorem}. We apply the combinatorial decomposition from Lemma~\ref{general-Hooley}, use Abel summation together with Conjecture~\ref{conjecture}, and evaluate the resulting sums by repeated contour shifts and residue calculations. The main terms are then combined to obtain the asymptotic formula in
Theorem~\ref{maintheorem}.
\subsection*{Acknowledgements.}
We would like to thank Professor Nathan Ng at the University of Lethbridge for making the valuable suggestion of using the multidimensional truncated Perron formula in Section \ref{sec:main-proof} and for many other useful discussions related to this work.

\section{Lemmas}\label{sec:lemmas}
We start with the following combinatorial lemma, which generalizes a lemma of Hooley \cite{HL} to the shifted divisor function
$\tau_{\mathcal I}(n)$.
\begin{lemma}\label{general-Hooley}
  Let $\mathcal{I}=\{\alpha_1,\alpha_2,\alpha_3\}$. If $n \leq X$, then
    \begin{equation*}
        \tau_{\mathcal{I}}(n)=\sum_{l_1l_2l_3=n}l_1^{-\alpha_1}l_2^{-\alpha_2}l_3^{-\alpha_3}= \Sigma_1(n)-\Sigma_2(n)+\Sigma_3(n),
    \end{equation*}
    where
    \begin{align*}
        \Sigma_1(n)&=\sum_{\substack{l_1l_2l_3=n\\ l_1l_2 \leq X^{2/3};\;l_1\leq X^{1/3}}}l_1^{-\alpha_1}l_2^{-\alpha_2}l_3^{-\alpha_3}+\sum_{\substack{l_1l_2l_3=n\\ l_2l_3 \leq X^{2/3};\;l_2\leq X^{1/3}}}l_1^{-\alpha_1}l_2^{-\alpha_2}l_3^{-\alpha_3}\\
        &\quad+\sum_{\substack{l_1l_2l_3=n\\ l_1l_3 \leq X^{2/3};\;l_3\leq X^{1/3}}}l_1^{-\alpha_1}l_2^{-\alpha_2}l_3^{-\alpha_3},\\
        \Sigma_2(n)&=\sum_{\substack{l_1l_2l_3=n\\ l_1l_2 \leq X^{2/3};\;l_1,l_3\leq X^{1/3}}}l_1^{-\alpha_1}l_2^{-\alpha_2}l_3^{-\alpha_3}+\sum_{\substack{l_1l_2l_3=n\\ l_2l_3 \leq X^{2/3};\;l_1,l_2\leq X^{1/3}}}l_1^{-\alpha_1}l_2^{-\alpha_2}l_3^{-\alpha_3}\\
        &\quad+\sum_{\substack{l_1l_2l_3=n\\ l_1l_3 \leq X^{2/3};\;l_2,l_3\leq X^{1/3}}}l_1^{-\alpha_1}l_2^{-\alpha_2}l_3^{-\alpha_3},\\
        \Sigma_3(n)&=\sum_{\substack{l_1l_2l_3=n\\ l_1,l_2,l_3  \leq X^{1/3}}}l_1^{-\alpha_1}l_2^{-\alpha_2}l_3^{-\alpha_3}.
    \end{align*}
\end{lemma}
\begin{proof}
We have the identity
    \begin{align*}
        \sum_{\substack{l_1l_2l_3=n\\ l_1,l_2,l_3 \leq X^{1/3}}}l_1^{-\alpha_1}l_2^{-\alpha_2}l_3^{-\alpha_3}&=\sum_{l_1l_2l_3=n}l_1^{-\alpha_1}l_2^{-\alpha_2}l_3^{-\alpha_3}-\sum_{1\leq i \leq 3} \sum_{\substack{l_1l_2l_3=n\\ l_i > X^{1/3}}}l_1^{-\alpha_1}l_2^{-\alpha_2}l_3^{-\alpha_3}\\
        &\quad+\sum_{1\leq i_1<i_2 \leq 3}\sum_{\substack{l_1l_2l_3=n\\ l_{i_1},l_{i_2} > X^{1/3}}}l_1^{-\alpha_1}l_2^{-\alpha_2}l_3^{-\alpha_3}-\sum_{\substack{l_1l_2l_3=n\\ l_1,l_2,l_3>X^{1/3}}}l_1^{-\alpha_1}l_2^{-\alpha_2}l_3^{-\alpha_3}.
    \end{align*}
Then,
    \begin{align*}
        \tau_{\mathcal{I}}(n)&=\sum_{1\leq i \leq 3} \sum_{\substack{l_1l_2l_3=n\\ l_i > X^{1/3}}}l_1^{-\alpha_1}l_2^{-\alpha_2}l_3^{-\alpha_3}-\sum_{1\leq i_1<i_2 \leq 3}\sum_{\substack{l_1l_2l_3=n\\ l_{i_1},l_{i_2} > X^{1/3}}}l_1^{-\alpha_1}l_2^{-\alpha_2}l_3^{-\alpha_3}\\
        &\quad+\sum_{\substack{l_1l_2l_3=n\\ l_1,l_2,l_3 \leq X^{1/3}}}l_1^{-\alpha_1}l_2^{-\alpha_2}l_3^{-\alpha_3}\\
        &=\sum_{\substack{l_1l_2l_3=n\\ l_1 > X^{1/3}}}l_1^{-\alpha_1}l_2^{-\alpha_2}l_3^{-\alpha_3}+\sum_{\substack{l_1l_2l_3=n\\ l_2 > X^{1/3}}}l_1^{-\alpha_1}l_2^{-\alpha_2}l_3^{-\alpha_3}+\sum_{\substack{l_1l_2l_3=n\\ l_3 > X^{1/3}}}l_1^{-\alpha_1}l_2^{-\alpha_2}l_3^{-\alpha_3}\\
        &\quad-\sum_{\substack{l_1l_2l_3=n\\ l_1,l_2 > X^{1/3}}}l_1^{-\alpha_1}l_2^{-\alpha_2}l_3^{-\alpha_3}-\sum_{\substack{l_1l_2l_3=n\\ l_2,l_3 > X^{1/3}}}l_1^{-\alpha_1}l_2^{-\alpha_2}l_3^{-\alpha_3}-\sum_{\substack{l_1l_2l_3=n\\ l_1,l_3 > X^{1/3}}}l_1^{-\alpha_1}l_2^{-\alpha_2}l_3^{-\alpha_3}\\
        &\quad+\sum_{\substack{l_1l_2l_3=n\\ l_1,l_2,l_3  \leq X^{1/3}}}l_1^{-\alpha_1}l_2^{-\alpha_2}l_3^{-\alpha_3}.
        \end{align*}
Combining the terms yields the desired decomposition, and hence proves the lemma.
\end{proof}

Next, we consider a Dirichlet series generated by $\tau_{\mathcal{J}}(n)$.

\begin{lemma}\label{lemma:Dirichlet}
Let $\mathcal{J}=\{\beta_1,\beta_2,\beta_3\}$, where
$|\beta_i|\leq\delta$ for $i=1,2,3$, and let $h\geq1$.
For $\Re(s)>1+\delta$, we have
\begin{align*}
\sum_{(n,h)=1}\frac{\tau_{\mathcal J}(n)}{n^s}
&=
\zeta(s+\beta_1)\zeta(s+\beta_2)\zeta(s+\beta_3)\\
&\qquad\times
\prod_{p\mid h}
\left(1-\frac{1}{p^{s+\beta_1}}\right)
\left(1-\frac{1}{p^{s+\beta_2}}\right)
\left(1-\frac{1}{p^{s+\beta_3}}\right).
\end{align*}
\end{lemma}

\begin{proof}
We know
\begin{align*}
    \sum_{n=1}^\infty \frac{\tau_{\mathcal{J}}(n)}{n^s}&=\zeta(s+\beta_1)\zeta(s+\beta_2)\zeta(s+\beta_3)\\
    &=\prod_p\left(1-\frac{1}{p^{s+\beta_1}}\right)^{-1}\left(1-\frac{1}{p^{s+\beta_2}}\right)^{-1}\left(1-\frac{1}{p^{s+\beta_3}}\right)^{-1}.
\end{align*}
Then,
\begin{align*}
    \sum_{(n,h)=1}\frac{\tau_{\mathcal{J}}(n)}{n^s}&=\prod_{p\nmid h} \left(1-\frac{1}{p^{s+\beta_1}}\right)^{-1}\left(1-\frac{1}{p^{s+\beta_2}}\right)^{-1}\left(1-\frac{1}{p^{s+\beta_3}}\right)^{-1}\\
    &=\prod_{p} \left(1-\frac{1}{p^{s+\beta_1}}\right)^{-1}\left(1-\frac{1}{p^{s+\beta_2}}\right)^{-1}\left(1-\frac{1}{p^{s+\beta_3}}\right)^{-1}\\
    &\qquad \times \prod_{p|h} \left(1-\frac{1}{p^{s+\beta_1}}\right)\left(1-\frac{1}{p^{s+\beta_2}}\right)\left(1-\frac{1}{p^{s+\beta_3}}\right)\\
    &=\zeta(s+\beta_1)\zeta(s+\beta_2)\zeta(s+\beta_3)\prod_{p|h}\left(1-\frac{1}{p^{s+\beta_1}}\right)\left(1-\frac{1}{p^{s+\beta_2}}\right)\left(1-\frac{1}{p^{s+\beta_3}}\right).
\end{align*}
\end{proof}
We record a subconvexity bound for $\zeta(s)$ on vertical lines, which will let us control the size of $\zeta(s)$ off the critical line during the contour shifts in Section \ref{sec:main-proof}.

\begin{lemma}\label{lem:weyl-bound}
Suppose that $\lambda\geq0$ is such that
\[
\left|\zeta\left(\frac12+it\right)\right|
\ll_\epsilon (1+|t|)^{\lambda+\epsilon}.
\]
The classical Weyl bound allows us to take $\lambda=\frac16$.
Moreover, for $\frac12\leq\sigma\leq1$, $|t|\geq1$, and any
$\epsilon>0$,
\[
|\zeta(\sigma+it)|
\ll_\epsilon
(1+|t|)^{2\lambda(1-\sigma)+\epsilon}.
\]
\end{lemma}

\begin{proof}
The choice $\lambda=\frac16$ follows from the classical Weyl bound
\cite{WeylB,LittlewoodB}. On the line $\Re(s)=1$, we have
$\zeta(1+it)\ll\log(2+|t|)$ for $|t|\geq1$. The stated estimate
then follows from the Phragm\'en--Lindel\"of principle applied in the strip $\frac12\leq\Re(s)\leq1$.
\end{proof}
Finally, since our sums are indexed over several variables at once, we will need a multidimensional version of the truncated Perron formula to express $D_{\mathcal{I},\mathcal{J}}(X,1)$ as a multiple contour integral.
\begin{lemma}\label{lem:truncated-perron}
Let $f:\mathbb{N}^k\rightarrow \mathbb{C}$ be an arithmetic function in $k$-variables, and let 
$(\sigma_{a_1},\dots,\sigma_{a_k})$ be a $k$-tuple of abscissas of absolute convergence of the associated Dirichlet
series \[F(w_1,\dots,w_k)=\sum_{l_1,\dots,l_k=1}^{\infty}\frac{f(l_1,\dots,l_k)}{l_1^{w_1}\cdots l_k^{w_k}}.\] 
We have for non-integral values $x_1,\dots,x_k\geq1$, $c_1 > \max(0, \sigma_{a_1} ),\dots, c_k > \max(0, \sigma_{a_k} )$, and $T_1,\ldots,T_k\geq1$, 
\begin{align*}
   &\left| \sum_{l_1\leq x_1,\dots,l_k\leq x_k}f(l_1,\dots,l_k)-\frac{1}{(2\pi i)^k}\int_{c_1-iT_1}^{c_1+iT_1}\dots\int_{c_k-iT_k}^{c_k+iT_k}F(w_1,\dots,w_k)\frac{x_{1}^{w_1}\cdots x_{k}^{w_k}}{w_1\cdots w_k}dw_k\cdots dw_1\right|\\&\leq x_1^{c_1}\cdots x_k^{c_k}\sum_{l_1,\dots,l_k=1}^{\infty}\frac{|f(l_1,\dots,l_k)|}{l_1^{c_1}\cdots l_k^{c_k}}Q\left(\frac{1}{\max(\pi T_1|\log(x_1/l_1)|,1)},\dots,\frac{1}{\max(\pi T_k|\log(x_k/l_k)|,1)}\right),
\end{align*}
where
\begin{equation*}
    Q(X_1,\dots,X_k)=(X_1+1)\cdots(X_k+1)-1.
\end{equation*}
\end{lemma}
\begin{proof} 
See \cite[Proposition 5, p. 251--252]{Balazard-et-al}.
\end{proof}

\section{The Full Main Term for $D_{\mathcal{I},\mathcal{J}}(X,1)$: Proof of Theorem~\ref{maintheorem}}
\label{sec:main-proof}

Throughout this section, let $\mathcal{I}=\{\alpha_1,\alpha_2,\alpha_3\}$ and
$\mathcal{J}=\{\beta_1,\beta_2,\beta_3\}$, where
\[\alpha_i\neq\alpha_j,\qquad \beta_i\neq\beta_j
\qquad (i\neq j),\]
and
\[0\neq |\alpha_i|,|\beta_i|\leq\delta,
\qquad 0<\delta<\frac{1}{44}.
\]
We assume Conjecture~\ref{conjecture} with $\Delta<\frac{61}{66}$ throughout.
\subsection{Initial decomposition}
\label{subsec:initial}

We begin by applying Lemma~\ref{general-Hooley} to the factor
$\tau_{\mathcal{I}}(n)$ in $D_{\mathcal{I},\mathcal{J}}(X,1)$.
After interchanging the order of summation, we obtain
\begin{equation}\label{eq:initial-decomposition}
D_{\mathcal{I},\mathcal{J}}(X,1)
=
\Sigma_{111}(X)+\Sigma_{112}(X)+\Sigma_{113}(X)
-\Sigma_{211}(X)-\Sigma_{212}(X)-\Sigma_{213}(X)
+\Sigma_{31}(X),
\end{equation}
where
\begin{align*}
\Sigma_{111}(X)&=
\sum_{l_1 \leq X^{1/3}}
\sum_{l_2 \leq X^{2/3}/l_1}
\sum_{l_3 \leq X/(l_1l_2)}
\tau_{\mathcal{J}}(l_1l_2l_3+1)
l_1^{-\alpha_1}l_2^{-\alpha_2}l_3^{-\alpha_3},\\
\Sigma_{112}(X)
&=\sum_{l_2 \leq X^{1/3}}\sum_{l_3 \leq X^{2/3}/l_2}
\sum_{l_1 \leq X/(l_2l_3)}\tau_{\mathcal{J}}(l_1l_2l_3+1)
l_1^{-\alpha_1}l_2^{-\alpha_2}l_3^{-\alpha_3},\\
\Sigma_{113}(X)
&=
\sum_{l_3 \leq X^{1/3}}
\sum_{l_1 \leq X^{2/3}/l_3}
\sum_{l_2 \leq X/(l_1l_3)}
\tau_{\mathcal{J}}(l_1l_2l_3+1)
l_1^{-\alpha_1}l_2^{-\alpha_2}l_3^{-\alpha_3},
\end{align*}
and
\begin{align*}
\Sigma_{211}(X)
&=\sum_{l_1 \leq X^{1/3}}
\sum_{l_2 \leq X^{2/3}/l_1}
\sum_{l_3 \leq X^{1/3}}
\tau_{\mathcal{J}}(l_1l_2l_3+1)
l_1^{-\alpha_1}l_2^{-\alpha_2}l_3^{-\alpha_3},\\
\Sigma_{212}(X)
&=\sum_{l_2 \leq X^{1/3}}
\sum_{l_3 \leq X^{2/3}/l_2}
\sum_{l_1 \leq X^{1/3}}
\tau_{\mathcal{J}}(l_1l_2l_3+1)
l_1^{-\alpha_1}l_2^{-\alpha_2}l_3^{-\alpha_3},\\
\Sigma_{213}(X)
&=\sum_{l_3 \leq X^{1/3}}
\sum_{l_1 \leq X^{2/3}/l_3}
\sum_{l_2 \leq X^{1/3}}
\tau_{\mathcal{J}}(l_1l_2l_3+1)
l_1^{-\alpha_1}l_2^{-\alpha_2}l_3^{-\alpha_3}.
\end{align*}
Finally,
\begin{equation*}
\Sigma_{31}(X)=\sum_{l_1 \leq X^{1/3}}\sum_{l_2 \leq X^{1/3}}
\sum_{l_3 \leq X^{1/3}}\tau_{\mathcal{J}}(l_1l_2l_3+1)
l_1^{-\alpha_1}l_2^{-\alpha_2}l_3^{-\alpha_3}.
\end{equation*}
For convenience, we also write
\[
\Sigma_{11}
=\Sigma_{111}+\Sigma_{112}+\Sigma_{113},
\qquad\Sigma_{21}=
\Sigma_{211}+\Sigma_{212}+\Sigma_{213}.
\]

\subsection{Calculation for $\Sigma_{11}(X)-\Sigma_{21}(X)$}
\label{subsec:sigma11}

We treat the three differences
\[\Sigma_{111}(X)-\Sigma_{211}(X),\qquad
\Sigma_{112}(X)-\Sigma_{212}(X),\qquad
\Sigma_{113}(X)-\Sigma_{213}(X).
\]
By symmetry, it suffices to evaluate
$\Sigma_{111}(X)-\Sigma_{211}(X)$ in detail. We have
\begin{align}\label{eq:Sigma111-211}
\Sigma_{111}(X)-\Sigma_{211}(X)
&=
\sum_{l_1\leq X^{1/3}}
\sum_{l_2\leq X^{2/3}/l_1}
l_1^{-\alpha_1}l_2^{-\alpha_2}
\sum_{X^{1/3}<l_3\leq X/(l_1l_2)}
\tau_{\mathcal J}(l_1l_2l_3+1)l_3^{-\alpha_3}.
\end{align}
Put $q=l_1l_2$ and
\[
A_q(t)=\sum_{l_3\leq t}\tau_{\mathcal J}(ql_3+1).
\]
Applying Abel's summation to the innermost sum in
\eqref{eq:Sigma111-211}, we obtain
\begin{align}\label{eq:Abel-paired}
\sum_{X^{1/3}<l_3\leq X/q}
\tau_{\mathcal J}(ql_3+1)l_3^{-\alpha_3}
&=\left(\frac{X}{q}\right)^{-\alpha_3}
A_q\left(\frac{X}{q}\right)
-X^{-\alpha_3/3}A_q(X^{1/3})+
\alpha_3\int_{X^{1/3}}^{X/q}
A_q(t)t^{-\alpha_3-1}\,dt.
\end{align}
After the substitution $n=ql_3+1$, we have
\begin{equation}\label{eq:Aq-progression}
A_q(t)
=
\sum_{\substack{n\leq qt+1\\n\equiv1\bmod q}}
\tau_{\mathcal J}(n)-1.
\end{equation}
The contribution of the term $-1$ in \eqref{eq:Aq-progression} to
\eqref{eq:Abel-paired} is
\[-\left(\frac{X}{q}\right)^{-\alpha_3}
+X^{-\alpha_3/3}
-\alpha_3\int_{X^{1/3}}^{X/q}t^{-\alpha_3-1}\,dt=0.
\]
Hence, in the remainder of the calculation, we may replace $A_q(t)$ by
\[\sum_{\substack{n\leq qt+1\\n\equiv1\bmod q}}
\tau_{\mathcal J}(n).
\]
Applying Conjecture~\ref{conjecture} with $X$ replaced by $X+1$, we obtain 
\begin{equation}\label{eq:conjecture}
\sum_{\substack{n\leq Y\\ n\equiv1\bmod q}}
\tau_{\mathcal J}(n)
=\frac{1}{\varphi(q)}
\sum_{\substack{n\leq Y\\ (n,q)=1}}
\tau_{\mathcal J}(n)
+E(Y;1,q),
\end{equation}
where
\begin{equation}\label{eq:maximal-conjecture}
\sum_{q\leq X^{2/3}}
\max_{1\leq Y\leq X+1}
|E(Y;1,q)|
\ll X^{\Delta+\epsilon}.
\end{equation}
This applies to the present range $q\leq X^{2/3}$ since $X^{2/3}<(X+1)^{2/3}$. Since $X^{1/3}\leq t\leq X/q$, we have
\[
qt+1\leq X+1.
\]
Hence, \eqref{eq:maximal-conjecture} applies with $Y=qt+1$
throughout the integral in \eqref{eq:Abel-paired}. Applying the truncated Perron formula together with Lemma~\ref{lemma:Dirichlet}, we have
\begin{align}\label{eq:Perron}
\frac{1}{\varphi(q)}
\sum_{\substack{n\leq Y\\(n,q)=1}}\tau_{\mathcal J}(n)
&=\frac{1}{\varphi(q)}
\frac{1}{2\pi i}
\int_{1+\delta+\epsilon-iT}^{1+\delta+\epsilon+iT}
\frac{Y^s}{s}F(s)\,ds
+\mathcal O\left(
\frac{Y^{1+\delta+\epsilon}}{\varphi(q)T}
\right)\nonumber\\
&=\frac{1}{\varphi(q)}
\operatorname*{Res}_{\substack{s=1-\beta_j\\j=1,2,3}}
\frac{Y^s}{s}F(s)
-L_1-L_2-L_3
+\mathcal O\left(
\frac{Y^{1+\delta+\epsilon}}{\varphi(q)T}
\right),
\end{align}
where
\[
F(s)=\zeta(s+\beta_1)\zeta(s+\beta_2)\zeta(s+\beta_3)f_q(s)
\]
and
\begin{equation}\label{eq:fq}
f_q(s)
=\prod_{p\mid q}
\left(1-p^{-(s+\beta_1)}\right)
\left(1-p^{-(s+\beta_2)}\right)
\left(1-p^{-(s+\beta_3)}\right).
\end{equation}
The three remaining contour integrals are
\begin{align*}
L_1&=\frac{1}{\varphi(q)}
\frac{1}{2\pi i}\int_{1+\delta+\epsilon+iT}^{1/2+\delta+iT}
\frac{Y^s}{s}F(s)\,ds,\\
L_2&=\frac{1}{\varphi(q)}
\frac{1}{2\pi i}
\int_{1/2+\delta+iT}^{1/2+\delta-iT}
\frac{Y^s}{s}F(s)\,ds,\\
L_3&=\frac{1}{\varphi(q)}
\frac{1}{2\pi i}
\int_{1/2+\delta-iT}^{1+\delta+\epsilon-iT}
\frac{Y^s}{s}F(s)\,ds.
\end{align*}
For $L_1$, Lemma~\ref{lem:weyl-bound}, with $\lambda=1/6$, gives
\begin{align*}
L_1
&\ll
\frac{q^\epsilon}{\varphi(q)}
\int_{1/2+\delta}^{1+\delta+\epsilon}
\frac{Y^\sigma}{T}
T^{\max\{1-\sigma+\delta,0\}+\epsilon}\,d\sigma\\
&\ll
\frac{q^\epsilon Y^{1+\delta+\epsilon}}
{\varphi(q)T}.
\end{align*}
The same estimate holds for $L_3$. Also,
\begin{align*}
L_2
&\ll
\frac{q^\epsilon Y^{1/2+\delta}}{\varphi(q)}
\int_{-T}^T
\frac{(1+|t|)^{1/2+\epsilon}}
{\sqrt{(1/2+\delta)^2+t^2}}\,dt\\
&\ll
\frac{q^\epsilon Y^{1/2+\delta}T^{1/2+\epsilon}}
{\varphi(q)}.
\end{align*}
Choosing $T=Y^{1/3}$, we obtain
\begin{equation}\label{eq:perron-error-corrected}
L_1+L_2+L_3
+
\mathcal O\left(
\frac{Y^{1+\delta+\epsilon}}{\varphi(q)T}
\right)
\ll
\frac{q^\epsilon Y^{2/3+\delta+\epsilon}}
{\varphi(q)}.
\end{equation}
Combining \eqref{eq:conjecture}, \eqref{eq:Perron}, and
\eqref{eq:perron-error-corrected}, we obtain
\begin{align}\label{eq:main-asymptotic}
\sum_{\substack{n\leq Y\\n\equiv1\bmod q}}
\tau_{\mathcal J}(n)
&=
\frac{1}{\varphi(q)}
\operatorname*{Res}_{\substack{s=1-\beta_j\\j=1,2,3}}
\frac{Y^s}{s}
\zeta(s+\beta_1)\zeta(s+\beta_2)\zeta(s+\beta_3)f_q(s)
\nonumber\\
&\quad+
\mathcal O\left(
\frac{q^\epsilon Y^{2/3+\delta+\epsilon}}
{\varphi(q)}
\right)
+E(Y;1,q).
\end{align}
Let
\[
\mathcal E_q=\max_{Y\leq X+1}|E(Y;1,q)|.
\]
Using \eqref{eq:maximal-conjecture}, the contribution of $E(Y;1,q)$
to \eqref{eq:Sigma111-211} is
\begin{equation*}
\ll X^{\delta+\epsilon}
\sum_{q\leq X^{2/3}}\mathcal E_q
\ll X^{\Delta+\delta+\epsilon}.
\end{equation*}
Similarly, the error term in \eqref{eq:main-asymptotic} contributes
\begin{equation*}
\ll X^{2/3+2\delta+\epsilon}.
\end{equation*}
We now evaluate the residues at $s=1-\beta_j$ in
\eqref{eq:main-asymptotic}.  We first use the binomial
expansion
\[(qt+1)^{1-\beta_j}=(qt)^{1-\beta_j}\left(
1+\sum_{n=1}^{\infty}\binom{1-\beta_j}{n}(qt)^{-n}
\right).\]
Since $qt\geq X^{1/3}$ for $X^{1/3}\leq t\leq X/q$, the terms with
$n\geq 1$ give
\begin{equation}\label{eq:binomial-error}
(qt+1)^{1-\beta_j}
=q^{1-\beta_j}t^{1-\beta_j}+\mathcal O\left((qt)^\delta\right),
\end{equation}
uniformly for $|\beta_j|\leq\delta$. We now estimate the error arising from replacing
$(qt+1)^{1-\beta_j}$ by $q^{1-\beta_j}t^{1-\beta_j}$ in
\eqref{eq:binomial-error}. Fix $l_1,l_2$ and put $q=l_1l_2$. The error
$\mathcal{O}((qt)^\delta)$ enters each occurrence of $A_q(t)$ in
\eqref{eq:Abel-paired}. At the upper endpoint $t=X/q$, its
contribution is
\[
\ll
\frac{q^{\delta+\epsilon}}{\varphi(q)}
\left(\frac{X}{q}\right)^{2\delta},
\]
while at the lower endpoint $t=X^{1/3}$ it is
\[
\ll
\frac{q^{\delta+\epsilon}}{\varphi(q)}
X^{2\delta/3}.
\]
The contribution from the integral in \eqref{eq:Abel-paired} is
\[
\ll
\frac{q^{\delta+\epsilon}}{\varphi(q)}
|\alpha_3|
\int_{X^{1/3}}^{X/q} t^{2\delta-1}\,dt
\ll
\frac{q^{\delta+\epsilon}}{\varphi(q)}
\left(
\left(\frac{X}{q}\right)^{2\delta}
+
X^{2\delta/3}
\right).
\]
Hence, for fixed $l_1,l_2$, the total error contributed to the
right-hand side of \eqref{eq:Abel-paired} is
\[
\ll
q^{-1+\delta+\epsilon}
\left(
\left(\frac{X}{q}\right)^{2\delta}
+
X^{2\delta/3}
\right),
\]
where we used $\varphi(q)^{-1}\ll q^{-1+\epsilon}$. We now insert this estimate into \eqref{eq:Sigma111-211}.
Since $q=l_1l_2$ and
\[
|l_1^{-\alpha_1}l_2^{-\alpha_2}|
\leq
l_1^\delta l_2^\delta
=
q^\delta,
\]
the total contribution of the error in
\eqref{eq:binomial-error} to
$\Sigma_{111}(X)-\Sigma_{211}(X)$ is
\begin{align*}
&\ll
X^{2\delta+\epsilon}
\sum_{q\leq X^{2/3}}\frac{d(q)}{q}
+
X^{2\delta/3+\epsilon}
\sum_{q\leq X^{2/3}}
\frac{d(q)}{q^{1-2\delta}}
\\&\ll
X^{2\delta+\epsilon}
+
X^{2\delta/3+\epsilon}X^{4\delta/3+\epsilon}
\ll X^{2\delta+\epsilon}.
\end{align*}
Thus, the error term in \eqref{eq:binomial-error} contributes
$\mathcal{O}(X^{2\delta+\epsilon})$ to
$\Sigma_{111}(X)-\Sigma_{211}(X)$.

For fixed $l_1,l_2$, the two endpoint terms in
\eqref{eq:Abel-paired} give
\[
\frac{\displaystyle\prod_{\substack{i=1\\i\neq j}}^3
\zeta(1-\beta_j+\beta_i)}
{\varphi(q)}
\frac{f_q(1-\beta_j)q^{1-\beta_j}}
{1-\beta_j}\left(\left(\frac{X}{q}\right)^{1-\beta_j-\alpha_3}
-X^{(1-\beta_j-\alpha_3)/3}
\right),
\]
while the integral in \eqref{eq:Abel-paired} gives
\begin{align*}
&\frac{\alpha_3}{\varphi(q)}
\frac{\displaystyle\prod_{\substack{i=1\\i\neq j}}^3
\zeta(1-\beta_j+\beta_i)}
{1-\beta_j}
f_q(1-\beta_j)q^{1-\beta_j}
\int_{X^{1/3}}^{X/q}t^{-\beta_j-\alpha_3}\,dt\\
&\qquad=\frac{\alpha_3}{\varphi(q)}
\frac{\displaystyle\prod_{\substack{i=1\\i\neq j}}^3
\zeta(1-\beta_j+\beta_i)}
{(1-\beta_j)(1-\beta_j-\alpha_3)}
f_q(1-\beta_j)q^{1-\beta_j}
\left(\left(\frac{X}{q}\right)^{1-\beta_j-\alpha_3}
-X^{(1-\beta_j-\alpha_3)/3}
\right).
\end{align*}
Since
\[
\frac{1}{1-\beta_j}+\frac{\alpha_3}
{(1-\beta_j)(1-\beta_j-\alpha_3)}
=\frac{1}{1-\beta_j-\alpha_3},
\]
the contribution of the residue at $s=1-\beta_j$ to
\eqref{eq:Abel-paired} is
\begin{align*}
&\frac{\displaystyle\prod_{\substack{i=1\\i\neq j}}^3
\zeta(1-\beta_j+\beta_i)}
{\varphi(q)(1-\beta_j-\alpha_3)}
f_q(1-\beta_j)q^{1-\beta_j}\left(\left(\frac{X}{q}\right)^{1-\beta_j-\alpha_3}
-X^{(1-\beta_j-\alpha_3)/3}
\right).
\end{align*}
Summing over $l_1,l_2$, we obtain
\begin{align}\label{eq:Sigma111-211-star}
\Sigma_{111}(X)-\Sigma_{211}(X)
&=\sum_{j=1}^3
\frac{\displaystyle\prod_{\substack{i=1\\i\neq j}}^3
\zeta(1-\beta_j+\beta_i)}
{1-\beta_j-\alpha_3}
\left(X^{1-\beta_j-\alpha_3}\Sigma_j^\star(X)
-X^{(1-\beta_j-\alpha_3)/3}\Sigma_j^{\star\star}(X)
\right)\nonumber\\&\quad+\mathcal O\left(X^{\Delta+\delta+\epsilon}+X^{2/3+2\delta+\epsilon}
\right),
\end{align}
where
\begin{align*}
\Sigma_j^\star(X)
&:=
\sum_{l_1\leq X^{1/3}}l_1^{\alpha_3-\alpha_1}
\sum_{l_2\leq X^{2/3}/l_1}l_2^{\alpha_3-\alpha_2}
\frac{f_{l_1l_2}(1-\beta_j)}{\varphi(l_1l_2)},\\
\Sigma_j^{\star\star}(X)
&:=
\sum_{l_1\leq X^{1/3}}l_1^{1-\beta_j-\alpha_1}
\sum_{l_2\leq X^{2/3}/l_1}l_2^{1-\beta_j-\alpha_2}
\frac{f_{l_1l_2}(1-\beta_j)}{\varphi(l_1l_2)}.
\end{align*}

\subsubsection*{Treatment of $\Sigma_j^\star$.}

We consider the Dirichlet series
\[
\sum_{l_1,l_2=1}^\infty
\frac{f_{l_1l_2}(1-\beta_j)}{\varphi(l_1l_2)}
\frac{1}{l_1^{w_1+w_2+\alpha_1-\alpha_3}}
\frac{1}{l_2^{w_2+\alpha_2-\alpha_3}}.
\]
By using multiplicativity and Euler's product formula, we have
\begin{equation*}
\begin{split}
&\sum_{l_1,l_2=1}^\infty
\frac{f_{l_1l_2}(1-\beta_j)}{\varphi(l_1l_2)}
\frac{1}{l_1^{w_1+w_2+\alpha_1-\alpha_3}}
\frac{1}{l_2^{w_2+\alpha_2-\alpha_3}}\\
&\qquad=
\prod_p
\sum_{j_1,j_2}
\frac{f_{p^{j_1+j_2}}(1-\beta_j)}
{\varphi(p^{j_1+j_2})}
\frac{1}
{p^{j_1(w_1+w_2+\alpha_1-\alpha_3)
+j_2(w_2+\alpha_2-\alpha_3)}}.
\end{split}
\end{equation*}
By \eqref{eq:fq} and the definition of $\varphi(n)$, we have
\begin{align*}
\begin{split}
&\sum_{j_1,j_2}
\frac{f_{p^{j_1+j_2}}(1-\beta_j)}
{\varphi(p^{j_1+j_2})}
\frac{1}
{p^{j_1(w_1+w_2+\alpha_1-\alpha_3)
+j_2(w_2+\alpha_2-\alpha_3)}}\\
&\qquad=1+\prod_{\substack{i=1\\i\neq j}}^3
\left(1-\frac{1}{p^{1-\beta_j+\beta_i}}\right)
\left(
\frac{1}{p^{1+w_1+w_2+\alpha_1-\alpha_3}-1}
+\frac{1}{p^{1+w_2+\alpha_2-\alpha_3}-1}\right.\\
&\hspace{14em}\left.
+\frac{1}{(p^{1+w_1+w_2+\alpha_1-\alpha_3}-1)
(p^{1+w_2+\alpha_2-\alpha_3}-1)}
\right).
\end{split}
\end{align*}
Applying the multidimensional truncated Perron formula in
Lemma~\ref{lem:truncated-perron}, we obtain
\begin{align}\label{eq:Perron1}
\Sigma_j^\star(X)
&=
\frac{1}{(2\pi i)^2}
\int_{2\delta+\epsilon-iT_1}^{2\delta+\epsilon+iT_1}
\int_{2\delta+\epsilon-iT_2}^{2\delta+\epsilon+iT_2}
\frac{X^{w_1/3}}{w_1}
\frac{X^{2w_2/3}}{w_2}
\nonumber\\
&\qquad\times
\zeta(1+w_1+w_2+\alpha_1-\alpha_3)
\zeta(1+w_2+\alpha_2-\alpha_3)
A_j^{(1)}(w_1,w_2)\,dw_2\,dw_1
\nonumber\\
&\quad+
O\left(\frac{X^{2\delta+\epsilon}}{T_1}\right)
+
O\left(\frac{X^{2\delta+\epsilon}}{T_2}\right).
\end{align}
where $A_j^{(1)}(w_1,w_2)$ is the arithmetic factor defined in \eqref{eq:A1-intro}.
\begin{lemma}\label{lem:holomorphic}
The function $A_j^{(1)}(w_1,w_2)$ is analytic in the region
\[
\Re(w_2)>
\Re(\alpha_3)-\Re(\alpha_2)-\frac12
\]
and
\[
\Re(w_1)>
\Re(\alpha_3)-\Re(\alpha_1)-\frac12-\Re(w_2).
\]
\end{lemma}

\begin{proof}
It follows from the definition of $A_j^{(1)}(w_1,w_2)$ that
\[
A_j^{(1)}(w_1,w_2)
=
\prod_p\left(1+\mathcal O\left(
p^{-2-\Re(w_1+w_2+\alpha_1-\alpha_3)}
+
p^{-2-\Re(w_2+\alpha_2-\alpha_3)}
\right)\right).
\]
The Euler product converges absolutely provided that
\[
\Re(w_2+\alpha_2-\alpha_3)>-\frac12
\]
and
\[
\Re(w_1+w_2+\alpha_1-\alpha_3)>-\frac12.
\]
Hence, $A_j^{(1)}(w_1,w_2)$ is analytic in the stated region.
\end{proof}
Fix $w_1$ with $\Re(w_1)=2\delta+\epsilon$. We shift the contour
in the $w_2$ integral in \eqref{eq:Perron1} to the vertical segment from
$\sigma_2-iT_2$ to $\sigma_2+iT_2$, where
\[
-\frac12+2\delta+\epsilon
<
\sigma_2
<
-4\delta-\epsilon.
\]
The poles of the integrand are at
\[
w_2=0,\qquad
w_2=\alpha_3-\alpha_2,\qquad
w_2=\alpha_3-\alpha_1-w_1.
\]
By Cauchy's residue theorem,
\begin{align*}
&\frac{1}{2\pi i}
\int_{2\delta+\epsilon-iT_2}^{2\delta+\epsilon+iT_2}
\frac{X^{2w_2/3}}{w_2}
\zeta(1+w_1+w_2+\alpha_1-\alpha_3)
\zeta(1+w_2+\alpha_2-\alpha_3)
A_j^{(1)}(w_1,w_2)\,dw_2\\
&=
\operatorname{Res}_{w_2=0}
\left(
\frac{X^{2w_2/3}}{w_2}
\zeta(1+w_1+w_2+\alpha_1-\alpha_3)
\zeta(1+w_2+\alpha_2-\alpha_3)
A_j^{(1)}(w_1,w_2)
\right)\\
&\quad+
\operatorname{Res}_{w_2=\alpha_3-\alpha_2}
\left(
\frac{X^{2w_2/3}}{w_2}
\zeta(1+w_1+w_2+\alpha_1-\alpha_3)
\zeta(1+w_2+\alpha_2-\alpha_3)
A_j^{(1)}(w_1,w_2)
\right)\\
&\quad+
\operatorname{Res}_{w_2=\alpha_3-\alpha_1-w_1}
\left(
\frac{X^{2w_2/3}}{w_2}
\zeta(1+w_1+w_2+\alpha_1-\alpha_3)
\zeta(1+w_2+\alpha_2-\alpha_3)
A_j^{(1)}(w_1,w_2)
\right)\\
&\quad-I_1-I_2-I_3,
\end{align*}
where
\[
I_1=
\frac{1}{2\pi i}
\int_{2\delta+\epsilon+iT_2}^{\sigma_2+iT_2}
\frac{X^{2w_2/3}}{w_2}
\zeta(1+w_1+w_2+\alpha_1-\alpha_3)
\zeta(1+w_2+\alpha_2-\alpha_3)
A_j^{(1)}(w_1,w_2)\,dw_2,
\]
\[
I_2=
\frac{1}{2\pi i}
\int_{\sigma_2+iT_2}^{\sigma_2-iT_2}
\frac{X^{2w_2/3}}{w_2}
\zeta(1+w_1+w_2+\alpha_1-\alpha_3)
\zeta(1+w_2+\alpha_2-\alpha_3)
A_j^{(1)}(w_1,w_2)\,dw_2,
\]
and
\[
I_3=
\frac{1}{2\pi i}
\int_{\sigma_2-iT_2}^{2\delta+\epsilon-iT_2}
\frac{X^{2w_2/3}}{w_2}
\zeta(1+w_1+w_2+\alpha_1-\alpha_3)
\zeta(1+w_2+\alpha_2-\alpha_3)
A_j^{(1)}(w_1,w_2)\,dw_2.
\]
By Lemma~\ref{lem:weyl-bound}, we have
\begin{align*}
I_1
&\ll
\int_{\sigma_2}^{2\delta+\epsilon}
\frac{X^{2\sigma/3}}{\sqrt{\sigma^2+T_2^2}}
(1+|\Im(w_1)+T_2|)^{-2\lambda\sigma+4\lambda\delta+\epsilon}
(1+T_2)^{-2\lambda\sigma+4\lambda\delta+\epsilon}
\,d\sigma\\
&\ll
T_2^{2\lambda|\sigma_2|+4\lambda\delta+\epsilon}
(T_1+T_2)^{2\lambda|\sigma_2|+4\lambda\delta+\epsilon}
T_2^{-1}
+
\frac{X^{4\delta/3+\epsilon}}{T_2}.
\end{align*}
Similarly,
\[
I_3
\ll
T_2^{2\lambda|\sigma_2|+4\lambda\delta+\epsilon}
(T_1+T_2)^{2\lambda|\sigma_2|+4\lambda\delta+\epsilon}
T_2^{-1}
+
\frac{X^{4\delta/3+\epsilon}}{T_2}.
\]
For $I_2$, Lemma~\ref{lem:weyl-bound} gives
\begin{align*}
I_2
&\ll
X^{2\sigma_2/3}
\int_{-T_2}^{T_2}
\frac{
(1+|\Im(w_1)+t|)^{2\lambda|\sigma_2|+4\lambda\delta+\epsilon}
(1+|t|)^{2\lambda|\sigma_2|+4\lambda\delta+\epsilon}
}
{\sqrt{\sigma_2^2+t^2}}\,dt\\
&\ll
X^{-2|\sigma_2|/3}
(T_1+T_2)^{2\lambda|\sigma_2|+4\lambda\delta+\epsilon}
T_2^{2\lambda|\sigma_2|+4\lambda\delta+\epsilon}.
\end{align*}
Therefore,
\begin{align}\label{eqn:error-hor-ver}
I_1+I_2+I_3
&\ll T_2^{2\lambda|\sigma_2|+4\lambda\delta+\epsilon}
(T_1+T_2)^{2\lambda|\sigma_2|+4\lambda\delta+\epsilon}
\left(T_2^{-1}+X^{-2|\sigma_2|/3}\right)+\frac{X^{4\delta/3+\epsilon}}{T_2}.
\end{align}
Next,
\begin{align*}
&\operatorname{Res}_{w_2=\alpha_3-\alpha_1-w_1}
\left(
\frac{X^{2w_2/3}}{w_2}
\zeta(1+w_1+w_2+\alpha_1-\alpha_3)
\zeta(1+w_2+\alpha_2-\alpha_3)
A_j^{(1)}(w_1,w_2)
\right)\\
&\quad=
\frac{X^{2(\alpha_3-\alpha_1-w_1)/3}}
{\alpha_3-\alpha_1-w_1}
\zeta(1-w_1+\alpha_2-\alpha_1)
A_j^{(1)}(w_1,\alpha_3-\alpha_1-w_1).
\end{align*}
Its contribution to \eqref{eq:Perron1} is
\begin{align}\label{eq:moving-residue-star}
\frac{1}{2\pi i}
\int_{2\delta+\epsilon-iT_1}^{2\delta+\epsilon+iT_1}
\frac{X^{w_1/3}}{w_1}
\frac{X^{2(\alpha_3-\alpha_1-w_1)/3}}
{\alpha_3-\alpha_1-w_1}
\zeta(1-w_1+\alpha_2-\alpha_1)
A_j^{(1)}(w_1,\alpha_3-\alpha_1-w_1)\,dw_1.
\end{align}
We move the contour in \eqref{eq:moving-residue-star} to
$\Re(w_1)=|\sigma_2|$. The possible poles of the integrand are at
\[
w_1=0,\qquad
w_1=\alpha_3-\alpha_1,\qquad
w_1=\alpha_2-\alpha_1.
\]
Since
\[
\Re(\alpha_3-\alpha_1),\,
\Re(\alpha_2-\alpha_1)\leq 2\delta
\]
and, by our choice of $\sigma_2$,
\[
|\sigma_2|>4\delta+\epsilon>2\delta+\epsilon,
\]
all these poles lie to the left of the initial contour
$\Re(w_1)=2\delta+\epsilon$. Hence no pole is crossed.
It follows that
\begin{align}\label{eqn:error-res-1}
&\frac{1}{2\pi i}
\int_{2\delta+\epsilon-iT_1}^{2\delta+\epsilon+iT_1}
\frac{X^{w_1/3}}{w_1}
\frac{X^{2(\alpha_3-\alpha_1-w_1)/3}}
{\alpha_3-\alpha_1-w_1}
\zeta(1-w_1+\alpha_2-\alpha_1)
A_j^{(1)}(w_1,\alpha_3-\alpha_1-w_1)\,dw_1
\nonumber\\
&\qquad\ll
X^{(4\delta-|\sigma_2|)/3+\epsilon}
+
X^{4\delta/3}
T_1^{-2+2\lambda|\sigma_2|+4\lambda\delta+\epsilon}.
\end{align}
For the residue at $w_2=0$, we obtain
\begin{align}\label{eqn:residue2}
\frac{1}{2\pi i}
\int_{2\delta+\epsilon-iT_1}^{2\delta+\epsilon+iT_1}
\frac{X^{w_1/3}}{w_1}
\zeta(1+w_1+\alpha_1-\alpha_3)
\zeta(1+\alpha_2-\alpha_3)
A_j^{(1)}(w_1,0)\,dw_1.
\end{align}
For the residue at $w_2=\alpha_3-\alpha_2$, we obtain
\begin{align}\label{eqn:residue1}
\frac{1}{2\pi i}
\int_{2\delta+\epsilon-iT_1}^{2\delta+\epsilon+iT_1}
\frac{X^{2(\alpha_3-\alpha_2)/3+w_1/3}}
{w_1(\alpha_3-\alpha_2)}
\zeta(1+w_1+\alpha_1-\alpha_2)
A_j^{(1)}(w_1,\alpha_3-\alpha_2)\,dw_1.
\end{align}
We now consider \eqref{eqn:residue2}. We move its contour to
$\Re(w_1)=-\frac12-\sigma_2+\epsilon$. By Cauchy's residue theorem, we have
\begin{align}\label{eq:residue2-shift}
&\frac{1}{2\pi i}
\int_{2\delta+\epsilon-iT_1}^{2\delta+\epsilon+iT_1}
\frac{X^{w_1/3}}{w_1}
\zeta(1+w_1+\alpha_1-\alpha_3)
\zeta(1+\alpha_2-\alpha_3)
A_j^{(1)}(w_1,0)\,dw_1
\nonumber\\
&=\operatorname{Res}_{w_1=0}
\left(
\frac{X^{w_1/3}}{w_1}
\zeta(1+w_1+\alpha_1-\alpha_3)
\zeta(1+\alpha_2-\alpha_3)
A_j^{(1)}(w_1,0)
\right)\nonumber\\
&\quad+
\operatorname{Res}_{w_1=\alpha_3-\alpha_1}
\left(
\frac{X^{w_1/3}}{w_1}
\zeta(1+w_1+\alpha_1-\alpha_3)
\zeta(1+\alpha_2-\alpha_3)
A_j^{(1)}(w_1,0)
\right)
-I_4-I_5-I_6,
\end{align}
where
\[
I_4=
\frac{1}{2\pi i}
\int_{2\delta+\epsilon+iT_1}^{-1/2-\sigma_2+\epsilon+iT_1}
\frac{X^{w_1/3}}{w_1}
\zeta(1+w_1+\alpha_1-\alpha_3)
\zeta(1+\alpha_2-\alpha_3)
A_j^{(1)}(w_1,0)\,dw_1,
\]
\[
I_5=
\frac{1}{2\pi i}
\int_{-1/2-\sigma_2+\epsilon+iT_1}^{-1/2-\sigma_2+\epsilon-iT_1}
\frac{X^{w_1/3}}{w_1}
\zeta(1+w_1+\alpha_1-\alpha_3)
\zeta(1+\alpha_2-\alpha_3)
A_j^{(1)}(w_1,0)\,dw_1,
\]
and
\[
I_6=
\frac{1}{2\pi i}
\int_{-1/2-\sigma_2+\epsilon-iT_1}^{2\delta+\epsilon-iT_1}
\frac{X^{w_1/3}}{w_1}
\zeta(1+w_1+\alpha_1-\alpha_3)
\zeta(1+\alpha_2-\alpha_3)
A_j^{(1)}(w_1,0)\,dw_1.
\]
By Lemma~\ref{lem:weyl-bound}, we have
\begin{equation}\label{eq:I456}
I_4+I_5+I_6
\ll
X^{2\delta/3+\epsilon}T_1^{-1}
+
X^{(|\sigma_2|-1/2)/3+\epsilon}
T_1^{\lambda-2\lambda|\sigma_2|
+4\lambda\delta+\epsilon}.
\end{equation}
For \eqref{eqn:residue1}, we move the contour to the same line. Thus
\begin{align}\label{eq:residue1-shift}
&\frac{1}{2\pi i}
\int_{2\delta+\epsilon-iT_1}^{2\delta+\epsilon+iT_1}
\frac{X^{2(\alpha_3-\alpha_2)/3+w_1/3}}
{w_1(\alpha_3-\alpha_2)}
\zeta(1+w_1+\alpha_1-\alpha_2)
A_j^{(1)}(w_1,\alpha_3-\alpha_2)\,dw_1
\nonumber\\
&=\operatorname{Res}_{w_1=0}
\left(
\frac{X^{2(\alpha_3-\alpha_2)/3+w_1/3}}
{w_1(\alpha_3-\alpha_2)}
\zeta(1+w_1+\alpha_1-\alpha_2)
A_j^{(1)}(w_1,\alpha_3-\alpha_2)
\right)\nonumber\\
&\quad+
\operatorname{Res}_{w_1=\alpha_2-\alpha_1}
\left(
\frac{X^{2(\alpha_3-\alpha_2)/3+w_1/3}}
{w_1(\alpha_3-\alpha_2)}
\zeta(1+w_1+\alpha_1-\alpha_2)
A_j^{(1)}(w_1,\alpha_3-\alpha_2)
\right)
-I_7-I_8-I_9,
\end{align}
where
\[
I_7=
\frac{1}{2\pi i}
\int_{2\delta+\epsilon+iT_1}^{-1/2-\sigma_2+\epsilon+iT_1}
\frac{X^{2(\alpha_3-\alpha_2)/3+w_1/3}}
{w_1(\alpha_3-\alpha_2)}
\zeta(1+w_1+\alpha_1-\alpha_2)
A_j^{(1)}(w_1,\alpha_3-\alpha_2)\,dw_1,
\]
\[
I_8=
\frac{1}{2\pi i}
\int_{-1/2-\sigma_2+\epsilon+iT_1}^{-1/2-\sigma_2+\epsilon-iT_1}
\frac{X^{2(\alpha_3-\alpha_2)/3+w_1/3}}
{w_1(\alpha_3-\alpha_2)}
\zeta(1+w_1+\alpha_1-\alpha_2)
A_j^{(1)}(w_1,\alpha_3-\alpha_2)\,dw_1,
\]
and
\[
I_9=
\frac{1}{2\pi i}
\int_{-1/2-\sigma_2+\epsilon-iT_1}^{2\delta+\epsilon-iT_1}
\frac{X^{2(\alpha_3-\alpha_2)/3+w_1/3}}
{w_1(\alpha_3-\alpha_2)}
\zeta(1+w_1+\alpha_1-\alpha_2)
A_j^{(1)}(w_1,\alpha_3-\alpha_2)\,dw_1.
\]
Again, by Lemma~\ref{lem:weyl-bound}, we get
\begin{equation}\label{eq:I789}
I_7+I_8+I_9
\ll
X^{2\delta+\epsilon}T_1^{-1}
+
X^{4\delta/3+(|\sigma_2|-1/2)/3+\epsilon}
T_1^{\lambda-2\lambda|\sigma_2|
+4\lambda\delta+\epsilon}.
\end{equation}
Let $E^{(1)}(X,T_1,T_2)$ denote the sum of the two Perron
truncation errors in \eqref{eq:Perron1} and the contributions
estimated in \eqref{eqn:error-hor-ver}, \eqref{eqn:error-res-1},
\eqref{eq:I456}, and \eqref{eq:I789}.
Then
\begin{align*}
E^{(1)}(X,T_1,T_2)
&\ll
X^{2\delta/3+\epsilon}
T_2^{2\lambda|\sigma_2|+4\lambda\delta+\epsilon}
(T_1+T_2)^{2\lambda|\sigma_2|+4\lambda\delta+\epsilon}
\left(
T_2^{-1}+X^{-2|\sigma_2|/3}
\right)
\\
&\quad+
X^{(4\delta-|\sigma_2|)/3+\epsilon}
+
X^{2\delta+\epsilon}
\left(T_1^{-1}+T_2^{-1}\right)
\\
&\quad+
X^{4\delta/3+(|\sigma_2|-1/2)/3+\epsilon}
T_1^{\lambda-2\lambda|\sigma_2|
+4\lambda\delta+\epsilon}
\\
&\quad+
X^{4\delta/3}
T_1^{-2+2\lambda|\sigma_2|
+4\lambda\delta+\epsilon}.
\end{align*}
Taking
\[
\lambda=\frac16,\qquad
|\sigma_2|=\frac5{22},\qquad
T_1=X^{1/11},\qquad
T_2=X^{5/33},
\]
which is compatible with the above range for $\sigma_2$ when
$0<\delta<1/44$ (after decreasing $\epsilon$ if necessary), gives
\begin{equation}\label{eqn:error-1}
E^{(1)}(X,T_1,T_2)
\ll
X^{-5/66+4\delta/3+\epsilon}.
\end{equation}
Consequently,
\begin{align}\label{eqn:M1-1-final}
\Sigma_j^\star(X)
&=
\zeta(1+\alpha_1-\alpha_3)
\zeta(1+\alpha_2-\alpha_3)
A_j^{(1)}(0,0)\nonumber\\
&\quad+
\frac{X^{(\alpha_3-\alpha_1)/3}}
{\alpha_3-\alpha_1}\zeta(1+\alpha_2-\alpha_3)A_j^{(1)}(\alpha_3-\alpha_1,0)\nonumber\\
&\quad+\frac{X^{2(\alpha_3-\alpha_2)/3}}
{\alpha_3-\alpha_2}
\zeta(1+\alpha_1-\alpha_2)
A_j^{(1)}(0,\alpha_3-\alpha_2)\nonumber\\
&\quad+
\frac{X^{(2\alpha_3-\alpha_1-\alpha_2)/3}}
{(\alpha_2-\alpha_1)(\alpha_3-\alpha_2)}
A_j^{(1)}(\alpha_2-\alpha_1,\alpha_3-\alpha_2)\nonumber\\
&\quad+\mathcal O\left(X^{-5/66+4\delta/3+\epsilon}
\right).
\end{align}

\subsubsection*{Treatment of $\Sigma_j^{\star\star}$.}

Applying the multidimensional truncated Perron formula to
$\Sigma_j^{\star\star}(X)$ gives
\begin{align}\label{eq:Perron2}
\Sigma_j^{\star\star}(X)
&=\frac{1}{(2\pi i)^2}
\int_{2\delta+\epsilon-iT_1}^{2\delta+\epsilon+iT_1}
\int_{1+2\delta+\epsilon-iT_2}^{1+2\delta+\epsilon+iT_2}
\frac{X^{w_1/3}}{w_1}
\frac{X^{2w_2/3}}{w_2}
\nonumber\\&\qquad\times
\zeta(w_1+w_2+\beta_j+\alpha_1)
\zeta(w_2+\beta_j+\alpha_2)
A_j^{(2)}(w_1,w_2)\,dw_2\,dw_1
\nonumber\\
&\quad+\mathcal O\left(
\frac{X^{2/3+2\delta+\epsilon}}{T_1}
\right)
+\mathcal O\left(
\frac{X^{2/3+2\delta+\epsilon}}{T_2}
\right),
\end{align}
where $A_j^{(2)}(w_1,w_2)$ is the arithmetic factor defined in \eqref{eq:A2-intro}.

\begin{lemma}\label{lem:holomorphic2}
The function $A_j^{(2)}(w_1,w_2)$ is analytic in the region
\[
\Re(w_2)>
\frac12-\Re(\alpha_2)-\Re(\beta_j)
\]
and
\[
\Re(w_1)>
\frac12-\Re(\alpha_1)-\Re(\beta_j)-\Re(w_2).
\]
\end{lemma}

\begin{proof}
This follows from Lemma~\ref{lem:holomorphic}, since
\[
A_j^{(2)}(w_1,w_2)
=
A_j^{(1)}(w_1,w_2-1+\alpha_3+\beta_j).
\]
\end{proof}
We shift the $w_2$ contour in \eqref{eq:Perron2} to
\[
\Re(w_2)=\frac12+2\delta+\epsilon.
\]
The poles of the integrand are at
\[
w_2=1-\beta_j-\alpha_2
\]
and
\[
w_2=1-\beta_j-w_1-\alpha_1.
\]
By Cauchy's residue theorem, the inner integral in \eqref{eq:Perron2}
is equal to
\begin{align}\label{eq:w2-shift-starstar}
&\operatorname{Res}_{w_2=1-\beta_j-\alpha_2}
\left(
\frac{X^{2w_2/3}}{w_2}
\zeta(w_1+w_2+\beta_j+\alpha_1)
\zeta(w_2+\beta_j+\alpha_2)
A_j^{(2)}(w_1,w_2)
\right)\nonumber\\
&\quad+
\operatorname{Res}_{w_2=1-\beta_j-w_1-\alpha_1}
\left(
\frac{X^{2w_2/3}}{w_2}
\zeta(w_1+w_2+\beta_j+\alpha_1)
\zeta(w_2+\beta_j+\alpha_2)
A_j^{(2)}(w_1,w_2)
\right)\nonumber\\
&\quad-J_1-J_2-J_3,
\end{align}
where
\[
J_1=\frac{1}{2\pi i}
\int_{1+2\delta+\epsilon+iT_2}^{1/2+2\delta+\epsilon+iT_2}
\frac{X^{2w_2/3}}{w_2}
\zeta(w_1+w_2+\beta_j+\alpha_1)
\zeta(w_2+\beta_j+\alpha_2)
A_j^{(2)}(w_1,w_2)\,dw_2,
\]
\[
J_2=\frac{1}{2\pi i}
\int_{1/2+2\delta+\epsilon+iT_2}^{1/2+2\delta+\epsilon-iT_2}
\frac{X^{2w_2/3}}{w_2}
\zeta(w_1+w_2+\beta_j+\alpha_1)
\zeta(w_2+\beta_j+\alpha_2)
A_j^{(2)}(w_1,w_2)\,dw_2,
\]
and
\[
J_3=\frac{1}{2\pi i}
\int_{1/2+2\delta+\epsilon-iT_2}^{1+2\delta+\epsilon-iT_2}
\frac{X^{2w_2/3}}{w_2}
\zeta(w_1+w_2+\beta_j+\alpha_1)
\zeta(w_2+\beta_j+\alpha_2)
A_j^{(2)}(w_1,w_2)\,dw_2.
\]
Using Lemma~\ref{lem:weyl-bound}, we obtain
\begin{align}\label{eq:error-J123}
J_1+J_2+J_3\ll X^{2/3}
T_2^{\lambda+4\lambda\delta+\epsilon}
(T_1+T_2)^{\lambda+4\lambda\delta+\epsilon}
\left(
T_2^{-1}X^{2\delta+\epsilon}+X^{-1/3+2\delta+\epsilon}
\right).
\end{align}
The residue at $w_2=1-\beta_j-\alpha_2$ is
\begin{align}\label{eqn:fixed-residue-B}
\frac{X^{2(1-\beta_j-\alpha_2)/3}}
{1-\beta_j-\alpha_2}
\frac{1}{2\pi i}
\int_{2\delta+\epsilon-iT_1}^{2\delta+\epsilon+iT_1}
\frac{X^{w_1/3}}{w_1}
\zeta(1+w_1+\alpha_1-\alpha_2)
A_j^{(2)}(w_1,1-\beta_j-\alpha_2)\,dw_1.
\end{align}
The residue at $w_2=1-\beta_j-w_1-\alpha_1$ is
\begin{align*}
&X^{2(1-\beta_j-\alpha_1)/3}\frac{1}{2\pi i}
\int_{2\delta+\epsilon-iT_1}^{2\delta+\epsilon+iT_1}
\frac{X^{-w_1/3}}
{(1-\beta_j-w_1-\alpha_1)w_1}\\
&\qquad\times
\zeta(1-w_1-\alpha_1+\alpha_2)
A_j^{(2)}(w_1,1-\beta_j-w_1-\alpha_1)\,dw_1.
\end{align*}
When moving the contour in the preceding integral to
$\Re(w_1)=\frac12-2\delta-\epsilon$,
no pole is crossed. Since
\[
\left|X^{2(1-\beta_j-\alpha_1)/3}\right|
\leq X^{2(1+2\delta)/3},
\]
Lemma~\ref{lem:weyl-bound} gives, on taking $T_1=X^{1/6}$,
\begin{align}\label{eq:error-moving-residue}
&X^{2(1-\beta_j-\alpha_1)/3}
\frac{1}{2\pi i}
\int_{2\delta+\epsilon-iT_1}^{2\delta+\epsilon+iT_1}
\frac{X^{-w_1/3}}
{(1-\beta_j-w_1-\alpha_1)w_1}\nonumber\\
&\qquad\times
\zeta(1-w_1-\alpha_1+\alpha_2)
A_j^{(2)}(w_1,1-\beta_j-w_1-\alpha_1)\,dw_1
\ll X^{1/2+2\delta+\epsilon}.
\end{align}
We now move the contour in \eqref{eqn:fixed-residue-B} to $\Re(w_1)=-\frac12+2\delta+\epsilon$. By Cauchy's residue theorem, we have
\begin{align}\label{eq:fixed-residue-w1-shift}
&\frac{X^{2(1-\beta_j-\alpha_2)/3}}
{1-\beta_j-\alpha_2}
\frac{1}{2\pi i}
\int_{2\delta+\epsilon-iT_1}^{2\delta+\epsilon+iT_1}
\frac{X^{w_1/3}}{w_1}
\zeta(1+w_1+\alpha_1-\alpha_2)
A_j^{(2)}(w_1,1-\beta_j-\alpha_2)\,dw_1
\nonumber\\
&=\frac{X^{2(1-\beta_j-\alpha_2)/3}}
{1-\beta_j-\alpha_2}
\zeta(1+\alpha_1-\alpha_2)
A_j^{(2)}(0,1-\beta_j-\alpha_2)\nonumber\\
&\quad+
\frac{X^{2(1-\beta_j-\alpha_2)/3}}
{1-\beta_j-\alpha_2}
\frac{X^{(\alpha_2-\alpha_1)/3}}
{\alpha_2-\alpha_1}
A_j^{(2)}(\alpha_2-\alpha_1,1-\beta_j-\alpha_2)
-J_4-J_5-J_6,
\end{align}
where
\[
J_4=
\frac{X^{2(1-\beta_j-\alpha_2)/3}}
{1-\beta_j-\alpha_2}
\frac{1}{2\pi i}
\int_{2\delta+\epsilon+iT_1}^{-1/2+2\delta+\epsilon+iT_1}
\frac{X^{w_1/3}}{w_1}
\zeta(1+w_1+\alpha_1-\alpha_2)
A_j^{(2)}(w_1,1-\beta_j-\alpha_2)\,dw_1,
\]
\[
J_5=
\frac{X^{2(1-\beta_j-\alpha_2)/3}}
{1-\beta_j-\alpha_2}
\frac{1}{2\pi i}
\int_{-1/2+2\delta+\epsilon+iT_1}^{-1/2+2\delta+\epsilon-iT_1}
\frac{X^{w_1/3}}{w_1}
\zeta(1+w_1+\alpha_1-\alpha_2)
A_j^{(2)}(w_1,1-\beta_j-\alpha_2)\,dw_1,
\]
and
\[
J_6=
\frac{X^{2(1-\beta_j-\alpha_2)/3}}
{1-\beta_j-\alpha_2}
\frac{1}{2\pi i}
\int_{-1/2+2\delta+\epsilon-iT_1}^{2\delta+\epsilon-iT_1}
\frac{X^{w_1/3}}{w_1}
\zeta(1+w_1+\alpha_1-\alpha_2)
A_j^{(2)}(w_1,1-\beta_j-\alpha_2)\,dw_1.
\]
By Lemma~\ref{lem:weyl-bound}, on taking $T_1=X^{1/6}$, we get
\begin{equation}\label{eq:J456}
J_4+J_5+J_6
\ll X^{19/36+2\delta+\epsilon}.
\end{equation}
Let $E^{(2)}(X,T_1,T_2)$ denote the sum of the two Perron
truncation errors in \eqref{eq:Perron2}, the contribution of
$J_1+J_2+J_3$ estimated in \eqref{eq:error-J123}, the contribution
estimated in \eqref{eq:error-moving-residue}, and the contribution
of $J_4+J_5+J_6$ estimated in \eqref{eq:J456}. Taking
\[
\lambda=\frac16,\qquad
T_1=X^{1/6},\qquad
T_2=X^{1/3+2\delta},
\]
these estimates give
\[
E^{(2)}(X,T_1,T_2)
\ll
X^{19/36+2\delta+\epsilon}.
\]
Hence,
\begin{align}\label{eqn:M1-2}
\Sigma_j^{\star\star}(X)
&=
\frac{X^{2(1-\alpha_2-\beta_j)/3}}
{1-\alpha_2-\beta_j}
\zeta(1+\alpha_1-\alpha_2)
A_j^{(2)}(0,1-\beta_j-\alpha_2)\nonumber\\
&\quad+
\frac{X^{(2-\alpha_1-\alpha_2-2\beta_j)/3}}
{(\alpha_2-\alpha_1)(1-\alpha_2-\beta_j)}
A_j^{(2)}
(\alpha_2-\alpha_1,1-\beta_j-\alpha_2)\nonumber\\
&\quad+
\mathcal O\left(
X^{19/36+2\delta+\epsilon}
\right).
\end{align}

\subsubsection*{Combination of $\Sigma_j^\star$ and
$\Sigma_j^{\star\star}$.}

Substituting \eqref{eqn:M1-1-final} and \eqref{eqn:M1-2} into
\eqref{eq:Sigma111-211-star}, and using
$\Delta<61/66$, we obtain

\begin{align*}
\Sigma_{111}(X)-\Sigma_{211}(X)
&=M_{\{\alpha_1,\alpha_2,\alpha_3\}}(X)
+\mathcal O\left(X^{61/66+10\delta/3+\epsilon}
\right),
\end{align*}
where
\begin{align}\label{eq:M-alpha}
&M_{\{\alpha_1,\alpha_2,\alpha_3\}}(X)\nonumber\\
&=\sum_{j=1}^3
\frac{
\displaystyle\prod_{\substack{i=1\\i\neq j}}^3
\zeta(1-\beta_j+\beta_i)}
{1-\beta_j-\alpha_3}
\Bigg[
X^{1-\beta_j-\alpha_3}
\Bigg(\zeta(1+\alpha_1-\alpha_3)
\zeta(1+\alpha_2-\alpha_3)
A_j^{(1)}(0,0)\nonumber\\
&\hspace{7em}
+\frac{X^{(\alpha_3-\alpha_1)/3}}
{\alpha_3-\alpha_1}
\zeta(1+\alpha_2-\alpha_3)
A_j^{(1)}(\alpha_3-\alpha_1,0)\nonumber\\
&\hspace{7em}
+\frac{X^{2(\alpha_3-\alpha_2)/3}}
{\alpha_3-\alpha_2}
\zeta(1+\alpha_1-\alpha_2)
A_j^{(1)}(0,\alpha_3-\alpha_2)\nonumber\\
&\hspace{7em}
+\frac{X^{(2\alpha_3-\alpha_1-\alpha_2)/3}}
{(\alpha_2-\alpha_1)(\alpha_3-\alpha_2)}
A_j^{(1)}
(\alpha_2-\alpha_1,\alpha_3-\alpha_2)
\Bigg)\nonumber\\
&\hspace{3em}
-X^{(1-\beta_j-\alpha_3)/3}
\Bigg(
\frac{X^{2(1-\alpha_2-\beta_j)/3}}
{1-\alpha_2-\beta_j}
\zeta(1+\alpha_1-\alpha_2)
A_j^{(2)}(0,1-\beta_j-\alpha_2)\nonumber\\
&\hspace{7em}
+
\frac{X^{(2-\alpha_1-\alpha_2-2\beta_j)/3}}
{(\alpha_2-\alpha_1)(1-\alpha_2-\beta_j)}
A_j^{(2)}
(\alpha_2-\alpha_1,1-\beta_j-\alpha_2)
\Bigg)
\Bigg].
\end{align}

The evaluations of
$\Sigma_{112}(X)-\Sigma_{212}(X)$ and
$\Sigma_{113}(X)-\Sigma_{213}(X)$ follow from the same argument by the
cyclic permutations
\[
(\alpha_1,\alpha_2,\alpha_3)
\longmapsto
(\alpha_2,\alpha_3,\alpha_1)
\]
and
\[
(\alpha_1,\alpha_2,\alpha_3)
\longmapsto
(\alpha_3,\alpha_1,\alpha_2),
\]
respectively. We therefore
define
\begin{equation}\label{eqn:M12}
M^{(12)}_{\mathcal I,\mathcal J}(X)
=
M_{\{\alpha_1,\alpha_2,\alpha_3\}}(X)
+
M_{\{\alpha_2,\alpha_3,\alpha_1\}}(X)
+
M_{\{\alpha_3,\alpha_1,\alpha_2\}}(X).
\end{equation}
Hence
\begin{equation}\label{eq:Sigma12-final}
\Sigma_{11}(X)-\Sigma_{21}(X)
=
M^{(12)}_{\mathcal I,\mathcal J}(X)
+
\mathcal O\left(
X^{61/66+10\delta/3+\epsilon}
\right).
\end{equation}

\subsection{Calculation for $\Sigma_{31}(X)$}

For fixed $l_1,l_2$, put $q=l_1l_2$ and
\[
A_q(t)=\sum_{l_3\leq t}\tau_{\mathcal J}(ql_3+1).
\]
By Abel summation, we get
\begin{align}\label{eq:Sigma31-Abel}
\sum_{l_3\leq X^{1/3}}
\tau_{\mathcal J}(ql_3+1)l_3^{-\alpha_3}
&=
X^{-\alpha_3/3}A_q(X^{1/3})
+\alpha_3\int_1^{X^{1/3}}
A_q(t)t^{-\alpha_3-1}\,dt.
\end{align}
After the substitution $n=ql_3+1$, we obtain
\[
A_q(t)=
\sum_{\substack{n\leq qt+1\\ n\equiv1\bmod q}}
\tau_{\mathcal J}(n)-1.
\]
The term $-1$ contributes $-1$ to the right-hand side of
\eqref{eq:Sigma31-Abel}. After summing over $l_1,l_2$, its contribution
is $\ll X^{2/3+2\delta/3+\epsilon}$.
Using \eqref{eq:maximal-conjecture} and the asymptotic formula
obtained in \eqref{eq:main-asymptotic}, with $Y=qt+1$, we have
\begin{align*}
\sum_{\substack{n\leq qt+1\\ n\equiv1\bmod q}}
\tau_{\mathcal J}(n)
&=\sum_{j=1}^3\frac{\displaystyle\prod_{\substack{i=1\\i\neq j}}^3\zeta(1-\beta_j+\beta_i)}{\varphi(q)(1-\beta_j)}
f_q(1-\beta_j)(qt+1)^{1-\beta_j}\\&\quad+
\mathcal O\left(\frac{q^\epsilon(qt+1)^{2/3+\delta+\epsilon}}
{\varphi(q)}\right)+E(qt+1;1,q).
\end{align*}
Using the binomial expansion,
\[
(qt+1)^{1-\beta_j}=(qt)^{1-\beta_j}
\left(1+\sum_{n=1}^{\infty}
\binom{1-\beta_j}{n}(qt)^{-n}\right),
\]
we have
\[(qt+1)^{1-\beta_j}=q^{1-\beta_j}t^{1-\beta_j}+\mathcal O\left((qt)^\delta\right).\]
The contribution of the terms arising from the error in this formula is $\ll X^{2\delta+\epsilon}$. Moreover, the contribution of the  error term in \eqref{eq:main-asymptotic} is $\ll X^{2/3+2\delta+\epsilon}$,
while \eqref{eq:maximal-conjecture} gives
\[
\ll X^{\Delta+\delta+\epsilon}.
\]
The contribution of the residue at $s=1-\beta_j$ to
\eqref{eq:Sigma31-Abel}, apart from the common factor
\[
\frac{\displaystyle\prod_{\substack{i=1\\i\neq j}}^3
\zeta(1-\beta_j+\beta_i)}
{\varphi(q)}
f_q(1-\beta_j),
\]
is
\begin{align*}
&\frac{q^{1-\beta_j}X^{(1-\beta_j-\alpha_3)/3}}
{1-\beta_j}
+
\frac{\alpha_3q^{1-\beta_j}}{1-\beta_j}
\int_1^{X^{1/3}}t^{-\beta_j-\alpha_3}\,dt\\
&\qquad=
\frac{q^{1-\beta_j}X^{(1-\beta_j-\alpha_3)/3}}
{1-\beta_j-\alpha_3}
-
\frac{\alpha_3q^{1-\beta_j}}
{(1-\beta_j)(1-\beta_j-\alpha_3)}.
\end{align*}
After restoring the common factor and the weights
$l_1^{-\alpha_1}l_2^{-\alpha_2}$, the contribution to $\Sigma_{31}(X)$ of the second
term on the right is bounded by
\begin{align*}
&\sum_{l_1,l_2\leq X^{1/3}}
\left|l_1^{-\alpha_1}l_2^{-\alpha_2}\right|
\left|
\frac{\prod_{i\neq j}\zeta(1-\beta_j+\beta_i)}
{\varphi(q)}
f_q(1-\beta_j)
\frac{\alpha_3q^{1-\beta_j}}
{(1-\beta_j)(1-\beta_j-\alpha_3)}
\right| \\
&\qquad\ll
\sum_{l_1,l_2\leq X^{1/3}}
(l_1l_2)^{2\delta+\epsilon}
\ll
X^{2/3+4\delta/3+\epsilon},
\end{align*}
where $q=l_1l_2$.
Consequently,
\begin{align}\label{eq:Sigma31}
\Sigma_{31}(X)
&=
\sum_{j=1}^3
\frac{X^{(1-\beta_j-\alpha_3)/3}
\displaystyle\prod_{\substack{i=1\\i\neq j}}^3
\zeta(1-\beta_j+\beta_i)}
{1-\beta_j-\alpha_3}
\sum_{l_1\leq X^{1/3}}l_1^{1-\beta_j-\alpha_1}
\sum_{l_2\leq X^{1/3}}l_2^{1-\beta_j-\alpha_2}
\frac{f_{l_1l_2}(1-\beta_j)}
{\varphi(l_1l_2)}
\nonumber\\
&\quad+
\mathcal O\left(
X^{\Delta+\delta+\epsilon}
+
X^{2/3+2\delta+\epsilon}
\right).
\end{align}
We now evaluate the double sum in \eqref{eq:Sigma31}. Applying
Lemma~\ref{lem:truncated-perron}, we obtain
\begin{align}\label{eq:Perron3}
\sum_{l_1\leq X^{1/3}}&l_1^{1-\beta_j-\alpha_1}
\sum_{l_2\leq X^{1/3}}l_2^{1-\beta_j-\alpha_2}
\frac{f_{l_1l_2}(1-\beta_j)}{\varphi(l_1l_2)}
\nonumber\\
&=
\frac{1}{(2\pi i)^2}
\int_{1+2\delta+\epsilon-iT_1}^{1+2\delta+\epsilon+iT_1}
\int_{1+2\delta+\epsilon-iT_2}^{1+2\delta+\epsilon+iT_2}
\frac{X^{w_1/3}}{w_1}
\frac{X^{w_2/3}}{w_2}
\nonumber\\
&\qquad\times
\zeta(w_1+\beta_j+\alpha_1)
\zeta(w_2+\beta_j+\alpha_2)
A_j^{(3)}(w_1,w_2)\,dw_2\,dw_1
\nonumber\\
&\quad+
\mathcal O\left(
\frac{X^{2/3+4\delta/3+\epsilon}}{T_1}
\right)
+
\mathcal O\left(
\frac{X^{2/3+4\delta/3+\epsilon}}{T_2}
\right),
\end{align}
where $A_j^{(3)}(w_1,w_2)$ is the arithmetic factor given in \eqref{eq:A3-intro}.

\begin{lemma}\label{lem:holomorphic3}
The function $A_j^{(3)}(w_1,w_2)$  is analytic in the region
\[
\Re(w_1)>\frac12-\Re(\alpha_1)-\Re(\beta_j)
\]
and
\[
\Re(w_2)>\frac12-\Re(\alpha_2)-\Re(\beta_j).
\]
\end{lemma}

\begin{proof}
Since
\[
A_j^{(3)}(w_1,w_2)
=
A_j^{(1)}
(w_1-w_2,w_2-1+\alpha_3+\beta_j),
\]
the result follows from Lemma~\ref{lem:holomorphic}.
\end{proof}
Fix $w_1$ with $\Re(w_1)=1+2\delta+\epsilon$. We move the $w_2$
contour in \eqref{eq:Perron3} to
\[
\Re(w_2)=\frac12+2\delta+\epsilon.
\]
The only pole crossed is at
\[
w_2=1-\beta_j-\alpha_2.
\]
Applying Cauchy's residue theorem yields
\begin{align*}
&\frac{1}{2\pi i}
\int_{1+2\delta+\epsilon-iT_2}^{1+2\delta+\epsilon+iT_2}
\frac{X^{w_2/3}}{w_2}
\zeta(w_2+\beta_j+\alpha_2)
A_j^{(3)}(w_1,w_2)\,dw_2\\
&=
\frac{X^{(1-\beta_j-\alpha_2)/3}}
{1-\beta_j-\alpha_2}
A_j^{(3)}(w_1,1-\beta_j-\alpha_2)
-K_1-K_2-K_3,
\end{align*}
where
\[
K_1=
\frac{1}{2\pi i}
\int_{1+2\delta+\epsilon+iT_2}^{1/2+2\delta+\epsilon+iT_2}
\frac{X^{w_2/3}}{w_2}
\zeta(w_2+\beta_j+\alpha_2)
A_j^{(3)}(w_1,w_2)\,dw_2,
\]
\[
K_2=
\frac{1}{2\pi i}
\int_{1/2+2\delta+\epsilon+iT_2}^{1/2+2\delta+\epsilon-iT_2}
\frac{X^{w_2/3}}{w_2}
\zeta(w_2+\beta_j+\alpha_2)
A_j^{(3)}(w_1,w_2)\,dw_2,
\]
and
\[
K_3=
\frac{1}{2\pi i}
\int_{1/2+2\delta+\epsilon-iT_2}^{1+2\delta+\epsilon-iT_2}
\frac{X^{w_2/3}}{w_2}
\zeta(w_2+\beta_j+\alpha_2)
A_j^{(3)}(w_1,w_2)\,dw_2.
\]
By Lemma~\ref{lem:weyl-bound}, we obtain
\[
K_1+K_3
\ll
X^{(1+2\delta)/3+\epsilon}
T_2^{-1+\lambda+\epsilon},
\]
and
\[
K_2\ll X^{1/6+2\delta/3+\epsilon}T_2^{\lambda+\epsilon}.
\]
Using these estimates in \eqref{eq:Perron3}, the terms containing
$K_1,K_2,K_3$ are
\begin{equation}\label{eq:error31-w2}
\ll X^{2/3+4\delta/3+\epsilon}
T_2^{-1+\lambda}+X^{1/2+4\delta/3+\epsilon}
T_2^\lambda.
\end{equation}
The residue at $w_2=1-\beta_j-\alpha_2$ gives
\begin{align}\label{eqn:residue31-w2}
&\frac{X^{(1-\beta_j-\alpha_2)/3}}
{1-\beta_j-\alpha_2}
\frac{1}{2\pi i}
\int_{1+2\delta+\epsilon-iT_1}^{1+2\delta+\epsilon+iT_1}
\frac{X^{w_1/3}}{w_1}
\zeta(w_1+\beta_j+\alpha_1)
\nonumber\\&\qquad\times
A_j^{(3)}(w_1,1-\beta_j-\alpha_2)\,dw_1.
\end{align}
We move this contour to $\Re(w_1)=\frac12+2\delta+\epsilon$.
The only pole crossed is at
\[
w_1=1-\beta_j-\alpha_1.
\]
Writing out the three remaining integrals exactly as above, and using
Lemma~\ref{lem:weyl-bound}, their sum is
\begin{equation}\label{eq:error31-w1}
\ll
X^{2/3+4\delta/3+\epsilon}
T_1^{-1+\lambda}
+
X^{1/2+4\delta/3+\epsilon}
T_1^\lambda.
\end{equation}
The residue is
\begin{align}\label{eqn:residue31-w1}
\frac{X^{(1-\beta_j-\alpha_2)/3}}
{1-\beta_j-\alpha_2}
\frac{X^{(1-\beta_j-\alpha_1)/3}}
{1-\beta_j-\alpha_1}
A_j^{(3)}
(1-\beta_j-\alpha_1,1-\beta_j-\alpha_2).
\end{align}
Let $E^{(3)}(X,T_1,T_2)$ denote the sum of the two Perron
truncation errors in \eqref{eq:Perron3}, the contribution from the
$w_2$ contour shift estimated in \eqref{eq:error31-w2}, and the
contribution from the subsequent $w_1$ contour shift estimated in
\eqref{eq:error31-w1}. Thus,
\begin{align*}
E^{(3)}(X,T_1,T_2)
\ll{}&
X^{2/3+4\delta/3+\epsilon}
\left(
T_1^{-1+\lambda}+T_2^{-1+\lambda}
\right)\\
&+
X^{1/2+4\delta/3+\epsilon}
\left(
T_1^\lambda+T_2^\lambda
\right)
+
X^{2/3+4\delta/3+\epsilon}
\left(
T_1^{-1}+T_2^{-1}
\right).
\end{align*}
Taking
\[
\lambda=\frac16,\qquad T_1=T_2=X^{1/6},
\]
gives
\begin{equation*}
E^{(3)}(X,T_1,T_2)
\ll X^{19/36+4\delta/3+\epsilon}.
\end{equation*}
Therefore,
\begin{align}\label{eq:Sigma31-final}
\Sigma_{31}(X)
&=
\sum_{j=1}^3
\frac{X^{1-\beta_j-\alpha_1/3-\alpha_2/3-\alpha_3/3}}
{(1-\beta_j-\alpha_1)(1-\beta_j-\alpha_2)(1-\beta_j-\alpha_3)}
\prod_{\substack{i=1\\i\neq j}}^3
\zeta(1-\beta_j+\beta_i)\nonumber\\
&\qquad\times
A_j^{(3)}(1-\beta_j-\alpha_1,1-\beta_j-\alpha_2)
+\mathcal O\left(
X^{\Delta+\delta+\epsilon}+X^{31/36+2\delta+\epsilon}
\right).
\end{align}
We denote the main term in \eqref{eq:Sigma31-final} by
\begin{equation}\label{eqn:M3}
M^{(3)}_{\mathcal I,\mathcal J}(X)
=
\sum_{j=1}^3
\frac{X^{1-\beta_j-\alpha_1/3-\alpha_2/3-\alpha_3/3}}
{(1-\beta_j-\alpha_1)(1-\beta_j-\alpha_2)(1-\beta_j-\alpha_3)}
\prod_{\substack{i=1\\i\neq j}}^3
\zeta(1-\beta_j+\beta_i)
A_j^{(3)}(1-\beta_j-\alpha_1,1-\beta_j-\alpha_2).
\end{equation}

\subsection{Completion of the Proof}

Combining \eqref{eq:Sigma12-final} and \eqref{eq:Sigma31-final}, we obtain
\begin{align*}
D_{\mathcal I,\mathcal J}(X,1)
&=
M_{\mathcal I,\mathcal J}^{(12)}(X)
+
M_{\mathcal I,\mathcal J}^{(3)}(X)+
\mathcal O\left(X^{61/66+10\delta/3+\epsilon}
+X^{\Delta+\delta+\epsilon}
+X^{31/36+2\delta+\epsilon}
\right).
\end{align*}
We now simplify the sum
$M_{\mathcal I,\mathcal J}^{(12)}(X)
+M_{\mathcal I,\mathcal J}^{(3)}(X)$. Fix $j\in\{1,2,3\}$. We first consider the term containing
$A_j^{(1)}(0,\alpha_3-\alpha_2)$ in
$M_{\{\alpha_1,\alpha_2,\alpha_3\}}(X)$. It is
\[
\frac{
X^{1-\beta_j-\alpha_3}
\prod_{\substack{i=1\\i\neq j}}^3\zeta(1-\beta_j+\beta_i)}
{1-\beta_j-\alpha_3}
\frac{X^{2(\alpha_3-\alpha_2)/3}}
{\alpha_3-\alpha_2}
\zeta(1+\alpha_1-\alpha_2)
A_j^{(1)}(0,\alpha_3-\alpha_2).
\]
The same zeta factor occurs in the term containing
$A_j^{(1)}(\alpha_2-\alpha_3,0)$ in
$M_{\{\alpha_3,\alpha_1,\alpha_2\}}(X)$, namely
\[
\frac{
X^{1-\beta_j-\alpha_2}
\prod_{\substack{i=1\\i\neq j}}^3\zeta(1-\beta_j+\beta_i)}
{1-\beta_j-\alpha_2}
\frac{X^{(\alpha_2-\alpha_3)/3}}
{\alpha_2-\alpha_3}
\zeta(1+\alpha_1-\alpha_2)
A_j^{(1)}(\alpha_2-\alpha_3,0),
\]
where $A_j^{(1)}$ in this expression is formed after the cyclic
permutation
$(\alpha_1,\alpha_2,\alpha_3)\mapsto
(\alpha_3,\alpha_1,\alpha_2)$. The third term that we need is the one containing
$A_j^{(2)}(0,1-\beta_j-\alpha_2)$ in
$M_{\{\alpha_1,\alpha_2,\alpha_3\}}(X)$:
\[
-\frac{
X^{(1-\beta_j-\alpha_3)/3}
\prod_{\substack{i=1\\i\neq j}}^3\zeta(1-\beta_j+\beta_i)}
{1-\beta_j-\alpha_3}
\frac{X^{2(1-\alpha_2-\beta_j)/3}}
{1-\alpha_2-\beta_j}
\zeta(1+\alpha_1-\alpha_2)
A_j^{(2)}(0,1-\beta_j-\alpha_2).
\]
By \eqref{eq:A2-intro}, we have
\[
A_j^{(2)}(0,1-\beta_j-\alpha_2)
=
A_j^{(1)}(0,\alpha_3-\alpha_2).
\]
We also check directly from \eqref{eq:A1-intro} that the
$A_j^{(1)}$ factor in the second expression is the same. With the original ordering of the $\alpha$'s, substituting
$(w_1,w_2)=(0,\alpha_3-\alpha_2)$ into \eqref{eq:A1-intro} gives
\[
w_1+w_2+\alpha_1-\alpha_3=\alpha_1-\alpha_2,
\qquad
w_2+\alpha_2-\alpha_3=0.
\]
For the cyclically permuted ordering
$(\alpha_3,\alpha_1,\alpha_2)$, substituting
$(w_1,w_2)=(\alpha_2-\alpha_3,0)$ gives the same two values in the opposite order. Substitution in \eqref{eq:A1-intro} therefore gives the same Euler product.

The three powers of $X$ are also the same; in each case the exponent is
\[
1-\beta_j-\frac{2\alpha_2+\alpha_3}{3}.
\]
We may therefore factor out
\[
X^{1-\beta_j-(2\alpha_2+\alpha_3)/3}
\zeta(1+\alpha_1-\alpha_2)
\prod_{\substack{i=1\\i\neq j}}^3
\zeta(1-\beta_j+\beta_i)
A_j^{(1)}(0,\alpha_3-\alpha_2).
\]
The sum of the remaining coefficients is
\[
\frac{1}
{(1-\beta_j-\alpha_3)(\alpha_3-\alpha_2)}
+
\frac{1}
{(1-\beta_j-\alpha_2)(\alpha_2-\alpha_3)}
-
\frac{1}
{(1-\beta_j-\alpha_3)(1-\beta_j-\alpha_2)}
=0.
\]
Hence these three terms cancel. Applying the same calculation after the other two cyclic permutations of
$(\alpha_1,\alpha_2,\alpha_3)$ gives two further cancellations.

It remains to consider the terms containing
$A_j^{(1)}(\alpha_2-\alpha_1,\alpha_3-\alpha_2)$ and
$A_j^{(2)}(\alpha_2-\alpha_1,1-\beta_j-\alpha_2)$ in
$M_{\{\alpha_1,\alpha_2,\alpha_3\}}(X)$, together with the analogous terms in
$M_{\{\alpha_2,\alpha_3,\alpha_1\}}(X)$ and
$M_{\{\alpha_3,\alpha_1,\alpha_2\}}(X)$. By \eqref{eq:A2-intro},
\[
A_j^{(2)}
(\alpha_2-\alpha_1,1-\beta_j-\alpha_2)
=
A_j^{(1)}
(\alpha_2-\alpha_1,\alpha_3-\alpha_2).
\]
Also, \eqref{eq:A3-intro} gives
\[
A_j^{(3)}
(1-\beta_j-\alpha_1,1-\beta_j-\alpha_2)
=
A_j^{(1)}
(\alpha_2-\alpha_1,\alpha_3-\alpha_2).
\]
For the last $A_j^{(1)}$ factor, the two expressions
$w_1+w_2+\alpha_1-\alpha_3$ and
$w_2+\alpha_2-\alpha_3$ appearing in
\eqref{eq:A1-intro} are both zero. Thus the same arithmetic factor occurs after each of the two cyclic permutations.

The six terms coming from the three cyclic copies of
$M_{\{\alpha_1,\alpha_2,\alpha_3\}}(X)$, together with the term in
$M_{\mathcal I,\mathcal J}^{(3)}(X)$, all contain
\[
X^{1-\beta_j-(\alpha_1+\alpha_2+\alpha_3)/3}
\prod_{\substack{i=1\\i\neq j}}^3
\zeta(1-\beta_j+\beta_i)
A_j^{(3)}
(1-\beta_j-\alpha_1,1-\beta_j-\alpha_2).
\]

It remains only to compare their coefficients. Put
$x_r=1-\beta_j-\alpha_r$ for $r=1,2,3$. For
$M_{\{\alpha_1,\alpha_2,\alpha_3\}}(X)$, the coefficients of the two terms under consideration are
\[
\frac{1}{x_3(x_1-x_2)(x_2-x_3)}
\qquad\text{and}\qquad
-\frac{1}{x_3x_2(x_1-x_2)}.
\]
Their sum is
\[
\frac{1}{x_2(x_1-x_2)(x_2-x_3)}.
\]
The two cyclic permutations give, respectively,
\[
\frac{1}{x_3(x_2-x_3)(x_3-x_1)}
\qquad\text{and}\qquad
\frac{1}{x_1(x_3-x_1)(x_1-x_2)}.
\]
Adding these three expressions gives $-\frac{1}{x_1x_2x_3}$.
On the other hand, by \eqref{eqn:M3}, the coefficient of the
corresponding term in $M_{\mathcal I,\mathcal J}^{(3)}(X)$ is
$\frac{1}{x_1x_2x_3}$. These terms therefore cancel as well.  After all these cancellations, the only term from
$M_{\{\alpha_1,\alpha_2,\alpha_3\}}(X)$ that remains in
$M_{\mathcal I,\mathcal J}^{(12)}(X)
+M_{\mathcal I,\mathcal J}^{(3)}(X)$ is
\[
\sum_{j=1}^3
\frac{X^{1-\beta_j-\alpha_3}}
{1-\beta_j-\alpha_3}
\prod_{\substack{i=1\\i\neq j}}^3
\zeta(1-\beta_j+\beta_i)
\zeta(1+\alpha_1-\alpha_3)
\zeta(1+\alpha_2-\alpha_3)
A_j^{(1)}(0,0).
\]
Together with the two expressions obtained by cyclically permuting
$(\alpha_1,\alpha_2,\alpha_3)$, this is precisely
$Main_{\mathcal I,\mathcal J}(X)$ as defined in
\eqref{eq:Main-final}. Hence,
\[M_{\mathcal I,\mathcal J}^{(12)}(X)
+M_{\mathcal I,\mathcal J}^{(3)}(X)
=Main_{\mathcal I,\mathcal J}(X).
\]
Finally, since $\Delta<61/66$,
\[
X^{\Delta+\delta+\epsilon}
+
X^{31/36+2\delta+\epsilon}
\ll
X^{61/66+10\delta/3+\epsilon}.
\]
Therefore,
\[
D_{\mathcal I,\mathcal J}(X,1)
=
Main_{\mathcal I,\mathcal J}(X)
+
\mathcal O\left(
X^{61/66+10\delta/3+\epsilon}
\right).
\]
This completes the proof of Theorem~\ref{maintheorem}.
\bibliographystyle{siam}
\bibliography{ref.bib}

\end{document}